\documentclass[12pt,reqno]{amsart}
\usepackage{amsthm, amsmath, amssymb, amscd, latexsym, multicol, verbatim, enumerate}
\usepackage[colorlinks=true, pdfstartview=FitV, linkcolor=blue, citecolor=blue, urlcolor=blue, breaklinks=true]{hyperref}

\usepackage[top=1.15in, bottom=1.15in, left=1.17in, right=1.17in]{geometry}
\usepackage[english]{babel}
\usepackage{tikz}
\usepackage[all]{xy}
\usepackage[xcolor]{changebar}

\newcounter{cnt}
 \makeatletter
\def\mydggeometry{\makeatletter\dg@YGRID=1\dg@XGRID=20\unitlength=0.003pt\makeatother}
\numberwithin{equation}{section}

\newtheorem{theorem}{Theorem}[section]
\newtheorem{corollary}[theorem]{Corollary}
\newtheorem{lemma}[theorem]{Lemma}
\newtheorem{proposition}[theorem]{Proposition}

\newtheorem{example}[theorem]{Example}
\newtheorem{definition}[theorem]{Definition}
\newtheorem{remark}[theorem]{Remark}
\newtheorem*{thm}{Theorem}

\newcommand{\g}{\mathfrak{g}}
\newcommand{\n}{\mathfrak{n}}

\newcommand{\Z}{\mathbb{Z}}

\newcommand{\C}{\mathbb{C}}

\newcommand{\ra}{\rightarrow}

\newcommand{\ts}{\otimes}

\newcommand{\wt}{\widetilde}

\newcommand{\supp}{\operatorname*{supp}}
\newcommand{\id}{\operatorname*{id}}

\newcommand{\bc}{\mathbb{C}}

\newcommand{\talpha}{{\widetilde{\alpha}}}
\newcommand{\tbeta}{{\widetilde{\beta}}}
\newcommand{\tlambda}{{\widetilde{\lambda}}}

\newcommand{\ttau}{{\widetilde{\tau}}}

\newcommand{\tPhi}{\widetilde{\Phi}}
\newcommand{\tLambda}{\widetilde{\Lambda}}
\newcommand{\tDelta}{\widetilde{\Delta}}

\newcommand{\tW}{\widetilde{W}}
\newcommand{\tI}{\widetilde{I}}
\newcommand{\myleq}{\preccurlyeq}
\newcommand{\pair}[2]{\langle{#1},{#2}^{\vee}\rangle}
\newcommand{\meet}{\mathrel{\text{\raisebox{0.1ex}{\scalebox{0.8}{$\wedge$}}}}}
\newcommand{\join}{\mathrel{\text{\raisebox{0.1ex}{\scalebox{0.8}{$\vee$}}}}}
\DeclareMathOperator{\height}{ht}
\newcommand{\nnabla}{\nabla_n}

\newcommand{\typeA}{\mathsf{A}}
\newcommand{\SL}{\mathsf{SL}}

\title{Degenerate Schubert Varieties in Type $\typeA$}

\author{Rocco Chiriv\`\i}
\address{Dipartimento di Matematica e Fisica ``Ennio De Giorgi'', Universit\`a del Salento, Lecce, Italy}
\email{rocco.chirivi@unisalento.it}

\author{Xin Fang}
\address{Mathematisches Institut, Universit\"at zu K\"oln, 50931, Cologne, Germany}
\email{xinfang.math@gmail.com}

\author{Ghislain Fourier}
\address{Lehrstuhl B f\"ur Mathematik, RWTH Aachen University, 52056, Aachen, Germany}
\email{fourier@mathb.rwth-aachen.de}



\begin{document}

\begin{abstract}
We introduce rectangular elements in the symmetric group. In the framework of PBW degenerations, we show that in type $\mathsf{A}$ the degenerate Schubert variety associated with a rectangular element  is indeed a Schubert variety in a partial flag variety of the same type with larger rank. Moreover, the degenerate Demazure module associated with a rectangular element is isomorphic to the Demazure module for this particular Schubert variety of larger rank. This generalises previous results by Cerulli Irelli, Lanini and Littelmann for the PBW degenerate flag variety.
\end{abstract}

\maketitle

\section{Introduction}

\subsection{Motivation}

In the framework of PBW filtrations on simple Lie algebras and PBW degenerations of their cyclic modules, Feigin \cite{Fei} introduced the PBW degenerate flag variety $\mathcal{F}^a$. In type $\mathsf{A}$, these projective varieties can be defined by a variation of the containing relations in the definition of the complete flag variety. Similar to the complete flag varieties, the degenerate flag varieties can be realised either in \cite{Fei} using PBW degenerated modules as a highest weight orbit of a degenerated algebraic group, or in \cite{CFR} as a quiver Grassmannian associated with the type $\mathsf{A}$ equi-oriented quiver and a fixed dimension vector. 

It turned out, that these projective varieties have rich geometric and combinatorial structures. For example, they are normal locally complete intersections and an analogous of the Borel-Weil-Bott theorem holds \cite{FeFi}; their torus fixed points are counted by the median Genocchi numbers \cite{Fei2}, \textit{etc}. Moreover, a monomial basis for the homogeneous coordinate ring of $\mathcal{F}^a$, parametrised by lattice points in a normal polyhedron, has been provided in \cite{FFL}, again influencing a whole new chapter on toric degenerations of the flag variety and beyond. 

It is quite surprising, that Cerulli Irelli and Lanini \cite{CL}, later also together with Littelmann \cite{CLL}, recovered the PBW degenerate flag variety (resp. the degenerate module) as a Schubert variety (resp. a Demazure module) in a partial flag variety of double rank. This certainly implies all the aforementioned geometric and combinatorial properties. One has actually all the tools from Schubert calculus on hand to compute singularities, cohomologies \textit{etc}. 

The concept of PBW degenerations can be applied to spherical varieties in general: one partially abelianizes the corresponding Lie algebra and algebraic group (see \cite{FFL, Fei}) and considers its orbit in the associated graded module. A natural class of spherical varieties is that of Schubert varieties and a first step towards the analysis of PBW degenerate Schubert varieties has been carried out in \cite{Fou}, where for a triangular element in the Weyl group, a monomial basis of the homogeneous coordinate ring has been provided. From a combinatorial point of view, this basis can be identified with lattice points on a face of the previous polyhedron. Our main motivation for this paper is to figure out whether one can identify these degenerate Schubert varieties again as Schubert varieties in a partial flag variety of type $\mathsf{A}$. 

\subsection{Main result}

The following main result of this paper, as a combination of Theorem \ref{Thm:Schubert} and Proposition \ref{Prop:IsoGeneral}, answers this question for a large class of Weyl group elements which will be introduced after the statement.

\begin{thm}
Let $\tau \in \mathfrak{S}_{n+1}$ be a rectangular element and $\lambda$ be a dominant integral weight for $\mathfrak{sl}_{n+1}$.
\begin{enumerate}
\item The PBW degenerate Demazure module $V^{-,a}_\tau(\lambda)$ (see Section \ref{Subsec:Demazure} for the definition) is isomorphic to the Demazure module $V_{\widetilde{\tau}}(\widetilde{\lambda})$, where $\widetilde{\tau} \in \mathfrak{S}_{2n}$ and $\widetilde{\lambda}$ is a dominant integral weight for $\mathfrak{sl}_{2n}$.
\item The degenerate Schubert variety $X_w^a(\lambda)$ (again see Section \ref{Subsec:Demazure}) is isomorphic to a Schubert variety $X_{\widetilde{w}}(\widetilde{\lambda})$ in a partial flag variety of $\mathfrak{sl}_{2n}$.
\end{enumerate} 
\end{thm}

We explain the terminology \textit{rectangular} in the theorem. Rectangular elements may be characterized in various way: via a property of the set of positive roots sent to negative roots as in Definition \ref{Def:Rectangular} using rectangular subsets, or via certain inequalities as in Proposition \ref{Prop:Rectangular}, or, finally, as permutations avoiding the four patterns $2413$, $2431$, $4213$ and $4231$ as proved in Corollary \ref{Corollary:patternAvoiding}. %An element $\tau \in \mathfrak{S}_{n+1}$ is called rectangular, if for any $1\leq i<k<j<\ell\leq n+1$, the conditions $\tau(i)>\tau(j)$ and $\tau(k)>\tau(\ell)$ are equivalent to $\tau(i)>\tau(\ell)$ and $\tau(k)>\tau(j)$.
For example, in $\mathfrak{S}_4$, there are 20 rectangular elements. We give an example in Remark \ref{Rmk:NonRect} of a non-rectangular element, showing that the above theorem does not hold by following the construction in this paper.

Since the longest element in the Weyl group of type $\mathsf{A}$ is a rectangular element, the results in \cite{CL} and \cite{CLL} in type $\mathsf{A}$ are special cases of the above theorem (Example \ref{Ex:w0}). The proof provided here is different from that in \cite{CL}, while partially follows the strategy of \cite{CLL}, indeed this last paper was the main inspiration for our combinatorial construction. Note that in \cite{CLL} everything is proved over $\Z$ while all our results and proofs are over $\C$; however our assumption is only for the sake of simplicity and we believe it should be straightforward to generalize our paper to $\Z$.
%The difference is explained in the beginning of Section \ref{Sec:GeneralCase}.

Parallel to the results in type $\mathsf{A}$, the theory of PBW filtration has been well understood for the symplectic Lie algebras \cite{FFL2, FeFiL}, we expect results parallel to this paper in the symplectic case, depending on the presence of symplectic triangular and rectangular elements and corresponding useful bases for the degenerate Demazure modules.

%A subset $A$ of positive roots is called rectangular if it is closed under join and meet of two of its roots  (see Definition~\ref{defn-join-meet} and the following examples). Any rectangular subset is also triangular (as defined in \cite{Fou} or Definition~\ref{defn-tri-rec}). For Weyl group element $w$, we set $N(w)$ to be image of the negative roots in the positive roots under $w^{-1}$. $w$ is then called rectangular (resp. triangular) if $N(w)$ is a rectangular (resp. triangular) subset. For $\mathfrak{sl}_4$, one has 22 triangular elements and 20 rectangular elements.

\subsection{Outline of proof}

We briefly explain the idea of the proof. For $\sigma\in\mathfrak{S}_{n+1}$, we denote $N(\sigma)$ to be the image of the negative roots in the positive roots under $\sigma^{-1}$. Let $\tau\in\mathfrak{S}_{n+1}$ be a rectangular element. We construct an element $\ttau\in\mathfrak{S}_{2n}$
such that for any $\alpha,\beta\in N(\ttau)$, $\alpha+\beta\notin N(\ttau)$, which implies that the Lie subalgebra of $\mathfrak{sl}_{2n}$ spanned by roots in $N(\ttau)$ is abelian. Moreover, for a given dominant integral weight $\lambda$ for $\mathfrak{sl}_{n+1}$, we provide a dominant integral weight $\widetilde{\lambda}$ for $\mathfrak{sl}_{2n}$ such that the scalar products of $\lambda$ with the roots of $N(\tau)$ coincide with the scalar products of $\widetilde{\lambda}$ with $N(\ttau)$.

In this sense, we have identified the abelianized Lie subalgebras of $\mathfrak{sl}_{n+1}$ generating the Demazure modules, with abelian Lie subalgebras of $\mathfrak{sl}_{2n}$. We then show that this induces an isomorphism between the PBW degenerated Demazure modules $V^{-,a}_\tau(\lambda)$ for $\mathfrak{sl}_{n+1}$ and the Demazure modules $V_{\ttau}(\widetilde{\lambda})$ of $\mathfrak{sl}_{2n}$, terminating the proof of our main theorem.

\subsection{Organisation of paper}

The paper is organised as follows: we start by giving a few more details on the main theorem in Section~\ref{sec-result}, then we introduce the combinatorics related to rectangular subsets in Section~\ref{Sec:Rectangular}. In Section~\ref{Sec:DoubleRank}, we build the bridge from $\mathfrak{sl}_{n+1}$ to $\mathfrak{sl}_{2n}$, followed by the proof of the main theorem in Section~\ref{Sec:Proof}. We conclude the paper with a few consequences and remarks in Section~\ref{sec-cons}.

\subsection{Acknowledgements}  

Part of the work was carried out during a research visit of X.F. to University of Hannover. He would like to thank Leibniz University of Hannover for the hospitality. The first author would like to thank Professor Andrea Maffei for various useful conversations.

\section{Statement of main result}\label{sec-result}

\subsection{Notations in Lie theory}

Let $G=\SL_{n+1}(\bc)$, $B$ be the Borel subgroup of $G$ consisting of upper triangular matrices, $T$ be the maximal torus contained in $B$ with normalizer $N(T)$ and let $W=N(T)/T$ be the Weyl group.

Let $\g=\mathfrak{sl}_{n+1}$ be the corresponding Lie algebra and $\g=\n^+\oplus\mathfrak{h}\oplus\n^-$ be the triangular decomposition with $\mathfrak{h}$ the Lie algebra of $T$ and $\mathfrak{b}^+:=\mathfrak{n}^+\oplus \mathfrak{h}$ the Lie algebra of $B$. We denote by $U(\g)$, $U(\n^-)$ and $U(\mathfrak{b}^+)$ the corresponding universal enveloping algebras.

The root system of $\g$ is denoted by $\Phi$; $\Phi^+$ (resp. $\Phi^-$) stands for the set of positive (resp. negative) roots. We choose the usual numbering of the simple roots $\alpha_1,\cdots,\alpha_n$ in $\mathfrak{h}^*$ so that
$$
\Phi^+=\{\alpha_{i,j}=\alpha_i+\alpha_{i+1}+\cdots+\alpha_{j}\mid 1\leq i\leq j\leq n\}.
$$
The reflection associated with a root $\alpha\in\Phi$ will be denoted by $s_\alpha$; if $\alpha=\alpha_i$, we will abbreviate $s_i:=s_{\alpha_i}$. Having fixed an enumeration of the simple roots, we may identify $W$ with $\mathfrak{S}_n$.

The \emph{support} $\supp(\alpha_{i,j})$ of the root $\alpha_{i,j}$ is defined to be the subset $\{i,i+1,\cdots,j\}$ of $[n]:=\{1,2,\ldots,n\}$. A subset $I\subseteq [n]$ is called connected, if it is an interval. For a subset $I$, we let $\Phi_I^+$ denote the set of positive roots in $\Phi^+$ supported in $I$.

For $\alpha\in\Phi^+$ we fix $e_\alpha$ and $f_\alpha$ to be basis elements in the root spaces of $\g$ corresponding to $\alpha$ and $-\alpha$. We choose $h_\alpha\in\mathfrak{h}$ such that $e_\alpha,f_\alpha,h_\alpha$ form a $\mathfrak{sl}_2$-triple in $\g$. When $\alpha=\alpha_{i,j}\in\Phi^+$, we will denote $e_{i,j}:=e_{\alpha_{i,j}}$ and $f_{i,j}:=f_{\alpha_{i,j}}$.

Let $\Lambda$ (resp. $\Lambda^+$) denote the weight lattice (resp. the monoid of the dominant weights in $\Lambda$) and $\varpi_1,\cdots,\varpi_n$ denote the fundamental weights.

Now, fix a weight $\lambda \in \Lambda^+$, denote by $V(\lambda)$ the irreducible representation of $G$ of highest weight $\lambda$ and let $v_\lambda\in V(\lambda)$ be a fixed highest weight vector. Let also $P_\lambda$ be the parabolic subgroup of $G$ containing $B$ and stabilizing the line $\bc\cdot v_\lambda$ in $V(\lambda)$. The Schubert variety $X_\tau^\lambda$ associated with an element $\tau\in W$ in the partial flag variety $G/P_\lambda$ is the closure $\overline{B.\tau P_\lambda}$ of the $B$--orbit of any lift of $\tau$ to $N(T)$.

The extremal weight space $V(\lambda)_{\tau(\lambda)}$ of weight $\tau(\lambda)$ is one-dimensional; we fix the non-zero basis element $v_{\tau(\lambda)}:=\tau\cdot v_\lambda$ in this weight space. The Demazure module $V_\tau(\lambda) := U(\mathfrak{b}^+)\cdot v_{\tau(\lambda)}$ is the cyclic $\mathfrak{b}^+$-module generated by $v_{\tau(\lambda)}$. The closure of the orbit through $v_{\tau(\lambda)}$ in the projective space $\mathbb{P}(V_\tau(\lambda))$
$$X_\tau(\lambda) := \overline{B\cdot [v_{\tau(\lambda)}]} \subseteq \mathbb{P}(V_\tau(\lambda))$$
is isomorphic to the Schubert variety $X_\tau^\lambda$. In particular for $\tau=w_0$ the longest element of $W$ we get an embedding of the partial flag variety $G/P_\lambda$.
\par
For $\tau\in W$, we denote $N(\tau) := \Phi^+\cap \tau^{-1}(\Phi^-)$ the set of positive roots sent to negative roots by $\tau$, and $P(\tau):=\Phi^+\cap \tau^{-1}(\Phi^+)$ the set of positive roots sent to positive roots by $\tau$. We denote $\mathfrak{n}_\tau^-$ the subspace of $\n$ consisting of root spaces associated with negative roots in $-N(\tau)$: this is a Lie subalgebra of $\mathfrak{n}$. In this case, we let $N_\tau^-$ be the corresponding algebraic group in $G$ having Lie algebra $\mathfrak{n}_\tau^-$: it is generated by root subgroups associated with roots in $-N(\tau)$.

Since $-N(\tau)= \tau^{-1}N(\tau^{-1})$ we have
$$
V_\tau(\lambda)=\tau(\tau^{-1}U(\mathfrak{n}^+)\tau) \cdot v_\lambda=\tau U(\mathfrak{n}_\tau^-) \cdot v_\lambda.
$$
We define $V_\tau^-(\lambda):=U(\mathfrak{n}_\tau^-)\cdot v_\lambda$ and call it $\tau^{-1}$-\emph{twisted Demazure module}.

Notice that $e_\alpha\cdot v_{\tau(\lambda)}=0$ unless $\alpha\in N(\tau^{-1})$. We endow $V_\tau(\lambda)$ with an $\mathfrak{n}_\tau^-$-module structure via twisting by $\tau$. Then as $\mathfrak{n}_\tau^-$-modules, $V^-_\tau(\lambda)$ is isomorphic to $V_\tau(\lambda)$.

We define the \emph{twisted Schubert variety} as a highest weight orbit:
$$X_\tau^-(\lambda):=\overline{N_\tau^-\cdot [v_\lambda]} \subseteq \mathbb{P}(V(\lambda)).$$ 
As projective varieties, $X_\tau^-(\lambda)$ is isomorphic to $X_\tau(\lambda)$. The Schubert varieties which will be studied in this paper are the twisted Schubert varieties.

\subsection{PBW degeneration}\label{Sec:PBW}

Let $\mathfrak{a}$ be a finite dimensional Lie algebra. The PBW filtration on $U(\mathfrak{a})$ is defined by: for $k\in\mathbb{N}$,
$$U(\mathfrak{a})_{\leq k}:=\mathrm{span}\{x_1\cdots x_s\mid x_1,\cdots,x_s\in\mathfrak{a}, s\leq k\}.$$

It endows $U(\mathfrak{a})$ with a filtered algebra structure. By the Poincar\'e-Birkhoff-Witt (PBW) theorem, the associated graded algebra is isomorphic to the symmetric algebra $S(\mathfrak{a})$.

Let $M$ be a cyclic $\mathfrak{a}$-module generated by $m$: that is to say, $M=U(\mathfrak{a})\cdot m$. The PBW filtration on $U(\mathfrak{a})$ endows $M$ with a filtered $U(\mathfrak{a})$-module structure by defining for $k\in\mathbb{N}$, $M_{\leq k}:=U(\mathfrak{a})_{\leq k}\cdot m$. We denote by $M^a$ the associated graded $S(\mathfrak{a})$-module. It is easy to see that $M^a$ is a cyclic $S(\mathfrak{a})$-module generated by $m^a$, the image of $m$ in $M^a$. The defining ideal $I^a(M)$ of $M^a$ in $S(\mathfrak{a})$ is defined to be the kernel of the following map
$$
S(\mathfrak{a})\longrightarrow S(\mathfrak{a})\cdot m=M^a,\ \ x\longmapsto x\cdot m^a.
$$

When $\mathfrak{a}=\mathfrak{n}^-$ and $M=V(\lambda)$ for $\lambda\in P^+$, the $S(\mathfrak{n}^-)$-module $V^a(\lambda)$ is studied in \cite{FFL}. A monomial basis of $V^a(\lambda)$, parametrized by the Feigin-Fourier-Littelmann-Vinberg polytope, is constructed in \cite{ABS}; and the defining ideal of the $S(\mathfrak{n}^-)$-module $V^a(\lambda)$ is made explicit.

A geometric interpretation of the PBW degeneration is constructed by Feigin in \cite{Fei}. We briefly recall his construction. Let $\mathfrak{n}^{-,a}$ be the unique abelian Lie algebra such that $U(\mathfrak{n}^{-,a})=S(\mathfrak{n})$; and 
$$N^{-,a}:=\exp(\mathfrak{n}^{-,a})\simeq\mathbb{G}_a^N$$ 
be the $N$-fold product of the additive group $\mathbb{G}_a$ where $N=\dim\n^-$. The degenerate flag variety is defined in \cite{Fei} as the closure of the highest weight orbit
$$\mathcal{F}^a(\lambda):=\overline{N^{-,a}\cdot [v_\lambda^a]}\subseteq \mathbb{P}(V^a(\lambda)).$$

As shown in the same paper, this projective variety is a flat degeneration of the flag variety
$$
\mathcal{F}(\lambda):=\exp(\mathfrak{g})\cdot [v_\lambda] = \overline{B\cdot[v_\lambda]}\subseteq \mathbb{P}(V(\lambda)).
$$
Moreover, in \cite{CL}, the authors realized $\mathcal{F}^a(\lambda)$ as a Schubert variety in a partial flag variety of $\SL_{2n}(\mathbb{C})$; furthermore, in \cite{CLL}, the PBW degenerate modules $V^a(\lambda)$ are realized as Demazure modules, and the defining ideal of the $S(\mathfrak{n}^-)$-module $V^a(\lambda)$ is recovered from that of the Demazure module.

\subsection{PBW degeneration of Demazure modules}\label{Subsec:Demazure}

The goal of this paper is to study another situation of the PBW degeneration for the Demazure module associated with certain elements $\tau\in W$. In the setting of the last subsection, we consider $\mathfrak{a}=\n_\tau^{-}$, and $M=V_\tau^-(\lambda)$ for $\lambda\in \Lambda^+$.

The construction in the last subsection gives an $\n_\tau^{-,a}$-module $V^{-,a}_\tau(\lambda)$.

Since $\n_\tau^-\subseteq \n^-$ and $V_\tau^-(\lambda)\subseteq V(\lambda)$, we first compare the PBW filtration on $V_\tau^-(\lambda)$ itself and the one inherited from $V(\lambda)$.

\begin{proposition}\label{Prop:DegSchubert}
For any $\lambda\in\Lambda^+$, $V^{-,a}_\tau(\lambda)\subseteq V^a(\lambda)$.
\end{proposition}

\begin{proof}
It suffices to show that the following filtrations on $V_\tau^-(\lambda)$ coincide:
\begin{enumerate}
\item the filtration obtained from the PBW filtration on $V(\lambda)$ by intersection;
\item the filtration arising from the PBW filtration on $U(\mathfrak{n}_\tau^-)$.
\end{enumerate}
We show that for any $k\in\mathbb{N}$,
$$U(\mathfrak{n}^-)_{\leq k}\cap U(\mathfrak{n}_\tau^-)=U(\mathfrak{n}_\tau^-)_{\leq k}.$$
The inclusion $\supseteq$ is clear by definition. To show the other inclusion, 
we fix a total order on a weight basis $x_1,\cdots,x_N$ of $\mathfrak{n}^-$, which induces a totally ordered basis of $\mathfrak{n}_\tau^-$. Take an element $x\in U(\mathfrak{n}^-)_{\leq k}\cap U(\mathfrak{n}_\tau^-)$ and write it in the fixed basis of $\mathfrak{n}^-$ using the PBW theorem, then every monomial appearing with a non-zero coefficient has PBW degree less or equal to $k$. This element is contained in $U(\mathfrak{n}_\tau^-)$, which means that one can rewrite $x$ in the induced PBW basis of $\mathfrak{n}_\tau^-$. 
According to \cite[Theorem 6.3]{GGP} (see also the statement \cite[(0.2)]{HO}), this rewriting procedure will not increase the PBW degree, showing that $x\in U(\mathfrak{n}_\tau^-)_{\leq k}$.
\end{proof}

Let $N_\tau^{-,a}:=\exp(\n_\tau^{-,a})\simeq\mathbb{G}_a^M$ where $M=\ell(\tau)$. According to this proposition, we can look the highest weight orbit closure
$$X_\tau^a(\lambda):=\overline{N_\tau^{-,a}\cdot [v_\lambda^a]}\subseteq\mathbb{P}(V_\tau^{-,a}(\lambda))$$
as a closed subvariety of $\mathcal{F}^a(\lambda)$. As their constructions are similar to the twisted Schubert varieties, we call them \emph{degenerate Schubert varieties} in degenerate flag varieties. One should compare this definition with the one given by the third author in \cite{Fou}.

We recall the definition of triangular elements introduced by the third author in \cite{Fou}; later, in Proposition \ref{Proposition:triangularAvoiding} we characterize triangular elements as the permutations avoiding the two patterns $4231$ and $2413$.

\begin{definition}
An element $\tau \in W\simeq\mathfrak{S}_n$ is called \emph{triangular}, if the following property is satisfied: for any $1\leq i<k\leq j<\ell$ with $\tau(i)>\tau(j)$, $\tau(k)>\tau(\ell)$, one has $\tau(i)>\tau(\ell)$ and $\tau(k)\geq \tau(j)$.
\end{definition}

In \cite{Fou}, for a triangular element $\tau\in W$, a monomial basis of $V_\tau^{-,a}(\lambda)$ is constructed in the spirit of \cite{FFL}, which is parametrised by a suitable face of the polytope described in \cite{FFL}.

The goal of this paper is to study for which triangular element $\tau\in W$, the degenerate Schubert variety is indeed a Schubert variety in a higher rank special linear group.

\subsection{Statement of the main result}\label{Subsec:mainResult}

Let $\widetilde{G} = \SL_{2n}$. When an object $X$ (roots, weights, subalgebras, Weyl group elements, etc...) is referred to $\widetilde{G}$, we write it as $\widetilde{X}$. 

In order to state the main result of this paper, we give here a direct definition of rectangular elements in the Weyl group $W$, later in \ref{Subsec:StructureOfRectangular} we will present a more convenient but equivalent definition and then in Corollary \ref{Corollary:patternAvoiding} we will see a characterization in terms of pattern avoidance. It is clear from the following definition that rectangular elements are triangular.

\begin{definition}
An element $\tau \in W\simeq\mathfrak{S}_{n+1}$ is called \emph{rectangular}, if for any $1\leq i<k<j<\ell\leq n+1$, the following conditions are equivalent:
\begin{enumerate}
\item $\tau(i)>\tau(j)$ and $\tau(k)>\tau(\ell)$;
\item $\tau(i)>\tau(\ell)$ and $\tau(k)>\tau(j)$.
\end{enumerate}
\end{definition}

The main result of the paper is the following:

\begin{theorem}\label{Thm:Schubert}
Let $\tau \in W$ be a rectangular element and $\lambda \in \Lambda^+$. Then there exists $\widetilde{\lambda} \in \widetilde{\Lambda}^+$ and $\widetilde{\tau} \in \widetilde{W}\simeq \mathfrak{S}_{2n}$ such that
\[
X^a_\tau(\lambda) \textrm{ is isomorphic to } X_{\widetilde{\tau}}(\widetilde{\lambda})
\]
as $\mathbb{G}^{M}_a$-projective varieties where $M=\ell(\tau)$.
\end{theorem}

The rest of the paper is devoted to the proof of this theorem. In Section \ref{Sec:Rectangular} we study the combinatorics of rectangular elements in terms of rectangular subsets. In Section \ref{Sec:DoubleRank} we give the explicit construction of the Weyl group element $\ttau$ and the weight $\tlambda$ appearing in Theorem \ref{Thm:Schubert}, and study the properties of these elements. The proof of Theorem \ref{Thm:Schubert} will be given in Section \ref{Sec:Proof}.

\section{Rectangular elements}\label{Sec:Rectangular}

\subsection{Triangular and rectangular subsets}

The dominance order in $\Lambda$ results in the following poset structure on $\Phi^+$: for $\alpha_{i,j}$, $\alpha_{k,\ell}\in\Phi^+$,
$$\alpha_{i,j}\geq\alpha_{k,\ell}\text{ precisely when }i\leq k\text{ and }j\geq \ell.$$

\begin{definition}\label{defn-join-meet}
Let $\alpha_{i,j},\alpha_{k,\ell}\in\Phi^+$ be two positive roots. 
\begin{enumerate}
\item Their \emph{join} is defined by:
$$\alpha_{i,j}\join \alpha_{k,\ell}:=\alpha_{\min(i,k),\max(j,\ell)};$$
it is the minimal positive root greater than both $\alpha_{i,j}$ and $\alpha_{k,\ell}$.
\item If $\supp(\alpha_{i,j})\cap\supp(\alpha_{k,\ell})=[i,j]\cap[k,\ell]$ is non void, or equivalently if $\max(i,k)\leq\min(j,\ell)$, then we say the \emph{meet} of the two roots exists, and define it by:
$$\alpha_{i,j}\meet \alpha_{k,\ell}:=\alpha_{\max(i,k),\min(j,\ell)}.$$
If the meet exists it is the maximal positive root smaller than both $\alpha_{i,j}$ and $\alpha_{k,\ell}$; otherwise no positive root is smaller than the two roots.
\end{enumerate}
\end{definition}

\begin{definition}\label{defn-tri-rec}
Let $A\subseteq\Phi^+$ be a subset.
\begin{enumerate}
\item The set $A$ is called a \emph{triangular} subset (see \cite{Fou}), if 
\begin{itemize}
\item[(R1)] for any $\alpha,\beta\in A$ such that $\supp(\alpha)\cup\supp(\beta)$ is a connected subset of $[n]$, we have $\alpha\join\beta\in A$; moreover, $\alpha\meet\beta\in A$ if it exists.
\end{itemize}
\item A triangular subset $A$ is called a \emph{rectangular} subset, if the following condition holds:
\begin{itemize}
\item[(R2)] if $\alpha,\beta\in\Phi^+$, $\alpha\meet\beta$ exists and $\alpha\join\beta, \alpha\meet\beta\in A$, then $\alpha,\beta\in A$.
\end{itemize}
\item A rectangular subset $A$ is said to be \emph{irreducible}, if the longest root $\theta$ in $\Phi^+$ is contained in $A$; otherwise it is \emph{reducible}.
\end{enumerate}
\end{definition}

\subsection{Structure of rectangular subsets}\label{Subsec:StructureOfRectangular}

We start by giving a characterisation of the rectangular subsets. The first result reduces this problem to irreducible ones.

\begin{proposition}\label{Prop:DecompRectangular}
Let $A\subseteq\Phi^+$ be a rectangular subset. Then there exist uniquely determined connected subsets $I_1,\cdots,I_r$ of $[n]$ such that
\begin{enumerate}
\item[(i)] any union of two or more of the subsets $I_1,\cdots,I_r$ is not connected;
\item[(ii)] $A_t:=A\cap\Phi_{I_t}^+$ is an irreducible rectangular subset of $\Phi_{I_t}^+$;
\item[(iii)] $A=A_1\sqcup A_2\sqcup\cdots\sqcup A_r$.
\end{enumerate}
\end{proposition}

\begin{proof}
Let $I_1,\cdots,I_r$ be the connected components of the union of $\supp(\alpha)$ with $\alpha$ running through $A$.

We verify that $A_t$ is an irreducible rectangular subset of $\Phi_{I_t}^+$. We fix $1\leq t\leq r$ and verify (R1) and (R2).

Let $\alpha,\beta\in A_t$ be such that $\supp(\alpha)\cup\supp(\beta)$ is connected. Clearly $\alpha\join\beta\in A$ being $A$ rectangular. Moreover, by definition $\supp(\alpha),\supp(\beta)\subseteq I_t$, and $\alpha\join\beta\in\Phi_{I_t}^+$ follows by $\supp(\alpha\join\beta)=\supp(\alpha)\cup\supp(\beta)$. So $\alpha\meet\beta\in A\cap\Phi_{I_t}^+=A_t$.

If $\alpha\meet\beta$ exists, $\supp(\alpha\meet\beta)=\supp(\alpha)\cap\supp(\beta)\subseteq I_t$ implies $\alpha\meet\beta\in\Phi_{I_t}^+$. But we have also $\alpha\meet\beta\in A$ since $A$ is rectangular, so $\alpha\meet\beta\in A_t$. We have thus proved that $A_t$ fulfils (R1).

For (R2) let $\alpha,\beta\in\Phi^+$ be such that $\alpha\meet\beta$ exists and $\alpha\join\beta,\alpha\meet\beta\in A_t$. Since $A_t\subseteq A$ and $A$ is rectangular, $\alpha,\beta\in A$. Moreover, $\alpha\join\beta\in A_t$ implies $\supp(\alpha)\cup\supp(\beta)\subseteq I_t$, which gives $\supp(\alpha)$, $\supp(\beta)\subseteq I_t$, hence $\alpha,\beta\in A_t$. So also (R2) is fulfilled.

It remains to show the irreducibility of $A_t$: the longest root $\theta_t\in\Phi_{I_t}^+\cap A = A_t$. Let $I_t=\{r,r+1,\cdots,s\}$ with $r\leq s$. Then for any $r\leq k\leq s$ there exists $\beta_k\in A_t$ such that $k\in\supp(\beta_k)\subseteq I_t$. Since $I_t$ is a connected component, the triangularity of $A_t$ implies that 
$$\theta_t=\beta_r\join\beta_{r+1}\join\cdots\join\beta_s\in A_t.$$

The part (iii) is clear. Finally the subsets $I_1,I_2,\ldots,I_r$ are uniquely determined by $A$ since they are the supports of the maximal elements in $A$.
\end{proof}

Given $A\subseteq\Phi^+$, we define two subsets
$$A^-:=A\cap \{\alpha_{1,1},\alpha_{1,2},\cdots,\alpha_{1,n}\},\ \ A^+:=A\cap \{\alpha_{1,n},\alpha_{2,n},\cdots,\alpha_{n,n}\}.$$

The following lemma is clear.
\begin{lemma}\label{Lem:Aplusminus}
The following properties holds:
\begin{enumerate}
\item for any $\gamma\in\Phi^+$, there exists a unique $\alpha\in\{\alpha_{1,1},\cdots,\alpha_{1,n}\}$ and a unique $\beta\in\{\alpha_{1,n},\cdots,\alpha_{n,n}\}$ such that $\gamma=\alpha\meet\beta$;
\item if $\alpha,\beta\in A^-$ (resp. $A^+$), then $\alpha\meet\beta,\alpha\join\beta\in A^-$ (resp. $A^+$);
\item for $\alpha,\alpha'\in\Phi^+$, if $\alpha\meet\alpha'\in A^-$ (resp. $A^+$), then $\alpha,\alpha'\in A^-$ (resp. $A^+$);
\item for $\alpha\in A^-$, $\beta\in A^+$, we have $\alpha\join\beta=\theta$;
\item if $\gamma=\alpha\meet\beta$, $\gamma'=\alpha'\meet\beta'$ with $\alpha,\alpha'\in A^-$ and $\beta,\beta'\in A^+$, then
$$\gamma\join\gamma'=(\alpha\join\alpha')\meet(\beta\join\beta').$$
\end{enumerate}
\end{lemma}

The following result characterises the irreducible rectangular subsets by $A^-$ and $A^+$.

\begin{proposition}\label{proposition_irreducibleRectangular}
A subset $A\subseteq\Phi^+$ containing $\theta$ is irreducible rectangular if and only if 
$$A=\{\alpha\meet\beta\mid \alpha\in A^-,\ \beta\in A^+,\ \alpha\meet\beta\text{ exists}\}.$$
In particular, if $A$ is irreducible rectangular and $\alpha,\beta\in A$, then $\alpha\join\beta\in A$.
\end{proposition}

\begin{proof}
Assume that $A$ is irreducible rectangular. We first prove that all existing meets $\alpha\meet\beta$, with $\alpha\in A^-$, $\beta\in A^+$, are in $A$: indeed, let $\alpha:=\alpha_{1,k}$ and $\beta:=\alpha_{\ell,n}$ and suppose that $\alpha\meet\beta$ exists, then $\ell\leq k$ and $\supp(\alpha)\cup\supp(\beta)=[n]$ is connected so that $\alpha\meet\beta\in A$ by (R1).

On the other hand, take $\gamma:=\alpha_{i,j}\in A$; since $\theta=\alpha_{1,n}\in A$, by (R2) we have $\alpha_{1,j}\in A^-$ and $\alpha_{i,n}\in A^+$, hence $\gamma=\alpha_{1,j}\meet\alpha_{i,n}$ is a meet of an element of $A^-$ and an element of $A^+$.

Conversely, assume that $A$ consists of all the elements of the form $\alpha\meet\beta$ for $\alpha\in A^-$ and $\beta\in A^+$. We want to verify (R1) and (R2) for $A$.

Assume that $\gamma=\alpha\meet\beta$, $\gamma'=\alpha'\meet\beta'$ with $\alpha,\alpha'\in A^-$ and $\beta,\beta'\in A^+$. By Lemma~\ref{Lem:Aplusminus}~(5), 
$$
\gamma\join\gamma'=(\alpha\join\alpha')\meet(\beta\join\beta')\in A.
$$
Moreover, if the meet $\gamma\meet\gamma'$ exists, it equals $(\alpha\meet\alpha')\meet(\beta\meet\beta')$, which is in $A$. We have proved (R1).

For (R2), assume that $\gamma,\gamma'\in\Phi^+$ are such that $\gamma\meet\gamma'$ exists and is in $A$. By Lemma~\ref{Lem:Aplusminus}~(1), we may assume that $\gamma=\alpha\meet\beta$ and $\gamma'=\alpha'\meet\beta'$ with $\alpha,\alpha'\in\{\alpha_{1,1},\cdots,\alpha_{1,n}\}$ and $\beta,\beta'\in\{\alpha_{1,n},\cdots,\alpha_{n,n}\}$.

Then $\gamma\meet\gamma'=(\alpha\meet\alpha')\meet(\beta\meet\beta')$. This is an element of $A$ and moreover $\alpha\meet\alpha'\in\{\alpha_{1,1},\ldots,\alpha_{1,n}\}$ and $\beta\meet\beta'\in\{\alpha_{1,n},\ldots,\alpha_{n,n}\}$, so we find that $\alpha\meet\alpha'\in A^-$ and $\beta\meet\beta'\in A^+$ by the uniqueness property stated in Lemma~\ref{Lem:Aplusminus}~(1). So, by Lemma~\ref{Lem:Aplusminus}~(3), $\alpha,\alpha'\in A^-$ and $\beta,\beta'\in A^+$. Hence $\gamma=\alpha\meet\beta$ and $\gamma'=\alpha'\meet\beta'$ are elements of $A$.

The final statement of the Proposition follows at once by Lemma~\ref{Lem:Aplusminus}~(5) and (2).
\end{proof}

\subsection{Rectangular elements}\label{section_rectangularElements}

We define rectangular elements in a similar way to the definition of triangular elements in \cite[Proposition 1]{Fou}.

\begin{definition}\label{Def:Rectangular}
An element $\tau\in W$ is called rectangular if $N(\tau)$ is a rectangular subset of $\Phi^+$.
\end{definition}

\begin{example}\label{Example:rectangularSL4}
Let $n=3$ and $G=\SL_4(\mathbb{C})$. The elements $s_1s_3s_2$ and $s_1s_2s_3s_2s_1$ are not triangular hence not rectangular. The elements $s_1s_2s_3s_2$ and $s_1s_3s_2s_1$ are triangular but not rectangular. There are 22 triangular elements and 20 rectangular elements in $\mathfrak{S}_4$. Among the rectangular elements, there exists a unique element $s_1s_3$ which is not irreducible. 

These non-rectangular elements can be arranged into a diamond in the Bruhat order:
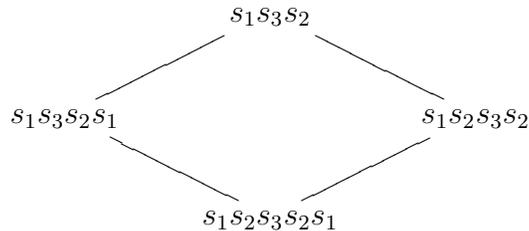
\[
\xymatrix{
	& s_1s_3s_2 \ar@{-}[dl]\ar@{-}[dr]\\
 s_1s_3s_2s_1 \ar@{-}[dr] & & s_1s_2s_3s_2 \ar@{-}[dl]\\
    & s_1s_2s_3s_2s_1
}
\]
%\centering        
%\begin{tikzpicture}[scale=.6]
%\node[right] at (0.5,0) {{$s_1s_3s_2$}};
%\node[right] at (3,2) {{$s_1s_2s_3s_2$}};
%\node[right] at (-3,2) {{$s_1s_3s_2s_1$}};
%\node[right] at (0,4) {{$s_1s_2s_3s_2s_1$}};
% \draw (2,0.5) -- (4.3,1.5);
% \draw (1,0.5) -- (-1.4,1.5);
% \draw (-1.4,2.5) -- (1,3.5);
% \draw (2,3.5) -- (4.3,2.5);
%\end{tikzpicture}
\end{example}

\begin{example}\label{Ex:running1}
Let $n=4$ and $\tau=s_1s_2s_3s_4s_1s_2s_1\in W$. Then
$$N(\tau)=\{\alpha_1,\alpha_{1,2},\alpha_2,\alpha_{1,4},\alpha_{2,4},\alpha_{3,4},\alpha_4\}\text{ and }N(\tau^{-1})=\{\alpha_1,\alpha_{1,2},\alpha_2,\alpha_{1,3},\alpha_{2,3},\alpha_3,\alpha_{1,4}\}.$$
It is clear that both $\tau$ and $\tau^{-1}$ are triangular elements in $W$, but only $\tau$ is a rectangular element. Indeed, for $\alpha_{1,3},\alpha_{2,4}\in \Phi^+$, 
$\alpha_{1,3}\meet\alpha_{2,4}=\alpha_{2,3}$ and $\alpha_{1,3}\join\alpha_{2,4}=\alpha_{1,4}$ are contained in  $N(\tau^{-1})$, but $\alpha_{2,4}\notin N(\tau^{-1})$.
\end{example}

\begin{remark}
The Kempf elements studied in \cite{Fou}, while being triangular, are not necessarily rectangular, for example $s_1s_2s_3s_2$ is a Kempf element but it is not rectangular.
\end{remark}

The following proposition proves that the definition of rectangular elements in \ref{Subsec:mainResult} is equivalent to the above one.
\begin{proposition}\label{Prop:Rectangular}
An element $\tau\in W$ is rectangular if and only if: for any $1\leq i< k < j < \ell\leq n$+1 the following two conditions are equivalent
\begin{itemize}
\item[(i)] $\tau(i)>\tau(j)$ and $\tau(k)>\tau(\ell)$
\item[(ii)] $\tau(i)>\tau(\ell)$ and $\tau(k)>\tau(j)$.
\end{itemize}
\end{proposition}

\begin{proof}
The proof is a direct consequence of the following observation: 
$$\alpha_{i,j-1}\in N(\tau)\text{ if and only if }\tau(i)>\tau(j).$$

Assume that $\tau\in W$ is a rectangular element, we prove the equivalence of (i) and (ii). For $1\leq i<k< j<\ell\leq n+1$, we define $\alpha = \alpha_{i,j-1}$, $\beta=\alpha_{k,\ell-1}\in\Phi^+$. Moreover $\alpha\join\beta=\alpha_{i,\ell-1}$ and $\alpha\meet\beta=\alpha_{k,j-1}$. In any case $\supp(\alpha)\cup\supp(\beta)=\{i,i+1,\ldots,\ell-1\}$ is connected.

So the condition (i) is equivalent to $\alpha,\beta\in N(\tau)$; condition (ii) is equivalent to: $\alpha\join\beta,\alpha\meet\beta\in N(\tau)$. Since $\supp(\alpha)\cup\supp(\beta)$ is connected, (i) implies (ii) by (R1) in the definition of rectangular elements and (ii) implies (i) by (R2).

Now assume that (i) and (ii) are equivalent and we show that $\tau$ is rectangular. Let $\alpha=\alpha_{i,j-1}$ and $\beta=\alpha_{k,\ell-1}$ for some $1\leq i<j\leq n+1$ and $1\leq k<\ell\leq n+1$; by symmetry we can assume that $i\leq k$. Note that if $i=k$ or $j\geq\ell$ then (R1) and (R2) are obvious, so we may assume that $i<k$ and $j<\ell$.

We prove (R1). So let $\alpha,\beta\in N(\tau)$ with $\supp(\alpha)\cup\supp(\beta)$ connected; this last condition forces $k\leq j$. If $k=j$ then $\alpha\join\beta=\alpha+\beta\in N(\tau)$ and $\alpha\meet\beta$ does not exist; if $k<j$ then (i) is fulfilled by $\alpha,\beta\in N(\tau)$, hence (ii) holds, and $\alpha\join\beta,\alpha\meet\beta\in N(\tau)$.

Finally we prove (R2). Let $\alpha=\alpha_{i,j-1},\beta=\alpha_{k,\ell-1}\in\Phi$ with $i\leq k$ be such that $\alpha\meet\beta$ exists and $\alpha\join\beta,\alpha\meet\beta\in N(\tau)$. The existence of $\alpha\meet\beta$ implies $k<j$, so $1\leq k<j<\ell$ and (ii) is fulfilled. Then (i) holds and $\alpha,\beta\in N(\tau)$. 
\end{proof}

The previous proposition has the following two corollaries.

\begin{corollary}
If $\tau\in W$ is a non--rectangular permutation and, in particular, $1\leq i<k<j<\ell\leq n+1$ are such that the condition of Proposition \ref{Prop:Rectangular} is violated by $\tau$, then all permutations in the parabolic cosets $\tau\cdot(\mathfrak{S}_t\times\mathfrak{S}_{n+1-t})$, for all $k\leq t\leq j$, are non-rectangular.
\end{corollary}
\begin{proof}
Let $a_h=\tau(h)$, for $h=1,2,\ldots,n+1$, and let ${\bf a}=a_1a_2\cdots a_{n+1}$ be the word representation of $\tau$. The word representation ${\bf a}'$ of a permutation $\tau'=\tau\eta$ with $\eta\in\mathfrak{S}_t\times\mathfrak{S}_{n+1-t}$ for a fixed $k\leq t\leq j$, is the word ${\bf a}$ up to a shuffling of the first $t$ entries and a shuffling of the last $n+1-t$ entries. In particular this may result in a possible swapping of $a_i$ and $a_k$ and a possible swapping of $a_j$ and $a_\ell$ in ${\bf a}'$. But any such swapping just exchange condition (i) and (ii) in Proposition \ref{Prop:Rectangular}, so $\tau'$ is still non-rectangular.
\end{proof}

\begin{corollary}\label{Corollary:patternAvoiding}
A permutation is rectangular if and only if it avoids the four patterns $2413$, $2431$, $4213$ and $4231$. The number $r_n$ of rectangular elements in $\mathfrak{S}_n$ is given by the recursively defined sequence: $r_1 = 1$, $r_2 = 2$ and $r_n = 4r_{n-1} - 2 r_{n-2}$ for any $r\geq 3$ (see entry {\tt A006012} in {\tt OEIS}, \cite{oeisRectangular}).
\end{corollary}
\begin{proof}
By Proposition \ref{Prop:Rectangular}, being rectangular is a property of the subwords of length $4$ of the one-line notation of a permutation; in particular, this property is expressed by some inequalities. Hence it may be expressed as a condition about patterns of length $4$. But the four patterns of the claim are the one-line notations of the four non-rectangular elements of $\mathfrak{S}_4$ as in Example \ref{Example:rectangularSL4}, so avoiding them is equivalent to being rectangular.

Now we prove the second statement about the number of rectangular permutations. As it is proved in \cite[Theorem~8]{biersAriel}, the recursively defined integer sequence as in the statement counts the number of permutations in $\mathfrak{S}_n$ avoiding the four patterns $1324$, $1423$, $2314$ and $2413$.

For a permutation $\tau$ let $i(\tau)$ be the permutation having as one-line notation the backward reading of the one-line notation of $\tau^{-1}$; the map $\tau\longmapsto i(\tau)$ is an involutive bijection of $\mathfrak{S}_n$. Moreover, if $\tau$ avoids the pattern $\sigma$ then $i(\tau)$ avoids the pattern $i(\sigma)$. This completes the proof since the patterns $2413$, $2431$, $4213$ and $4231$ are obtained by the patterns $1324$, $1423$, $2314$ and $2413$ by applying $i$.
\end{proof}

For completeness we prove that also triangular elements are characterized by pattern avoidance.
\begin{proposition}\label{Proposition:triangularAvoiding}
A permutation is triangular if and only if it avoids the two patterns $4231$ and $2413$. The number of triangular elements in $\mathfrak{S}_n$ is counted by {\tt A032351} in {\tt OEIS}, \cite{oeisTriangular}. 
\end{proposition}
\begin{proof}
The characterization of triangular elements as the permutations avoiding the two patterns $4231$ and $2413$ is proved as in the previous Corollary \ref{Corollary:patternAvoiding}; indeed note that $k\leq j$ is equivalent to $k<j$ in the definition of triangular element since for $k=j$ the condition there is always true.

Now we prove the statement about the number of triangular elements. The permutations avoiding the patterns $4231$, $2413$ are the same number of the permutations avoiding the patterns $4231$ and $3142$ by taking the inverse. This second set is in bijection with the set of permutations avoiding $3412$ and $4231$ by \cite[End of Section 1]{bona}. 
These are the smooth permutations counted by {\tt A032351} in {\tt OEIS}.
 \end{proof}

\section{From $\SL_{n+1}$ to $\SL_{2n}$}\label{Sec:DoubleRank}

Recall that $\wt{G}:=\SL_{2n}$, and $\wt{\Phi}$ (resp. $\wt{\Phi}^+$, $\wt{\Phi}^-$, $\wt{W}$, $\wt{\Lambda}^+$, etc...) is the root system (resp. set of positive roots, set of negative roots, Weyl group, weight lattice, etc...) of $\wt{G}$.

\subsection{Special elements in the Weyl group of $\SL_{2n}$}

The following subset of positive roots will play a central role in what follows

$$\nnabla=\big\{\talpha_{i,j}\mid 1\leq i\leq n\leq j\leq 2n-1,\,j-i\leq n-1\big\}.$$

Let $\myleq$ be a total order on $\tPhi^+$ such that for any $\alpha,\beta\in\tPhi^+$, $\height(\alpha) \leq \height(\beta)$ implies $\alpha\myleq\beta$. For a subset $A=\{\beta_r\myleq\beta_{r-1}\myleq\cdots\myleq\beta_1\}$ of $\tPhi^+$, we define
\[
\tau_A = s_{\beta_1} s_{\beta_2} \cdots s_{\beta_r}\in\tW\simeq\mathfrak{S}_{2n}.
\]

The following lemma will be repeatedly used in subsequent proofs.

\begin{lemma}\label{lemma_nablaKey}
If $\alpha = \talpha_{n-k+1,n+h-1}$ and $\beta=\talpha_{n-v+1,n+u-1}$ are elements of $\nnabla$ then
\[
s_\beta(\alpha) =
\left\{
\begin{array}{ll}
\alpha & \textrm{if }k\neq v\textrm{ and }h\neq u\\
\alpha-\beta & \textrm{if }(k=v\textrm{ and } h\neq u) \textrm{ or }(k\neq v \textrm{ and } h=u)\\
-\alpha & \textrm{if }k=v \textrm{ and }h=u.\\
\end{array}
\right.
\]
\end{lemma}

\begin{proof}
First notice that for any two roots $\gamma,\delta\in\nnabla$, $\gamma+\delta$ is not a root, since otherwise it would have a coefficient $2$ in front of $\alpha_n$ when it is written into simple roots. 

Let $a := \pair{\alpha}{\beta}$. Since $\wt{G}$ is of type $\typeA$, the possible values of $a$ are $0,\pm 1$ and $2$. The above argument shows that $a\neq -1$. The lemma follows from the following observations:
\begin{enumerate}
\item $a=2$ if and only if $\alpha=\beta$, \textit{i.e.}, $k=v$ and $h=u$;
\item if $\alpha\neq\beta$, $a=0$ if and only if $k\neq v$, $h\neq u$.
\end{enumerate}
\end{proof}

A first consequence of this lemma is the following proposition showing that the order $\myleq$ is not essential in the definition of $\tau_A$.

\begin{proposition}\label{proposition_orderIndependence}
Let $A\subseteq\nnabla$ be a subset. Then $\tau_A$ is independent of the total order $\myleq$, once it refines the height function.
\end{proposition}

\begin{proof} 
It suffices to show that for $\alpha\neq\beta\in\nnabla$ with $\height(\alpha)=\height(\beta)$, $\pair{\alpha}{\beta}=0$. We take $\alpha = \talpha_{n-k+1,n+h-1}$ and $\beta = \talpha_{n-v+1,n+u-1}$, then 
$$h+k-1 = \height(\alpha) = \height(\beta) = u+v-1.$$ 
Therefore $k\neq v$ is equivalent to $h\neq u$, and then $\pair{\alpha}{\beta}=0$ by Lemma \ref{lemma_nablaKey}.
\end{proof}

\begin{example}\label{Ex:running2}
\begin{enumerate}
\item Let $n=4$ and $A=\{\talpha_4,\talpha_{3,4},\talpha_{4,5},\talpha_{2,4},\talpha_{3,5},\talpha_{4,6},\talpha_{4,7}\}\subseteq \tPhi^+$. Then $\tau_A=s_4s_5s_6s_7s_3s_4s_2\in \tW$.
\item Let $A=\nnabla$. It is easy to see by induction that
$$\tau_A=(s_ns_{n+1}\cdots s_{2n-1})(s_{n-1}s_{n-2}\cdots s_{2n-3})\cdots (s_2s_3)s_1.$$
\end{enumerate}
\end{example}

\subsection{$\nnabla$-ideals}\label{section_nablaIdeal}

\begin{definition}
A subset $A\subseteq\nnabla$ is called a $\nnabla$--\emph{ideal}, if $\alpha\in A$, $\beta\in\nnabla$ and $\beta\leq\alpha$ imply that $\beta\in A$. 
\end{definition}

In the following we fix a $\nnabla$-ideal $A$. In order to investigate the properties of the element $\tau_A$, we introduce some combinatorial data attached to $A$. 

\begin{enumerate}
\item For $1\leq k\leq n$, we define $c_k(A)$ either as the maximum $h$ such that $\talpha_{n-k+1,n+h-1}$ is in $A$, or as $0$ if no such element is in $A$;
\item For $1\leq h \leq n$, we define $r_h(A)$ either as the maximum $k$ such that $\talpha_{n-k+1,n+h-1}$ is in $A$, or as $0$ if no such element is in $A$.
\end{enumerate} 

Both sequences are non-increasing: $c_1(A)\geq c_2(A)\geq \cdots \geq c_n(A)$ and $r_1(A)\geq r_2(A)\geq\cdots\geq r_n(A)$.

When the set $A$ is clear from the context, we will drop it from the notation.

\begin{example}\label{Ex:running3}
In the following picture we arrange the positive roots as the corresponding root subspaces in a matrix of the special linear group. The elements of $\tPhi^+$ not in $\nnabla$ are illustrated by circles; the elements of $\nnabla$ are labelled by squares. We fix a $\nnabla$--ideal $A\subseteq\tPhi^+$ as in Example \ref{Ex:running2} (1), whose elements are depicted as black squares. 

Notice that $\alpha\leq\beta$ if and only $\alpha$ is to the left and/or down with respect to $\beta$.
 
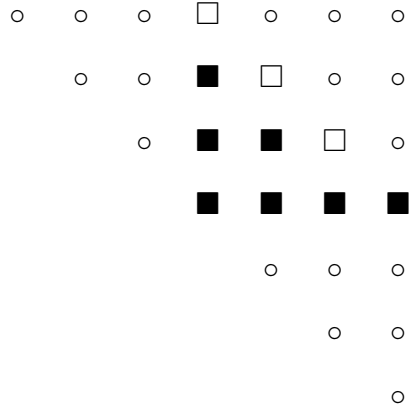
\[
\xymatrix@!C=0pt@!R=0pt{
 \circ & \circ & \circ & \square & \circ & \circ & \circ\\
      &\circ & \circ & \blacksquare & \square & \circ & \circ\\
      &        &\circ & \blacksquare & \blacksquare & \square & \circ\\
      &       &        &\blacksquare & \blacksquare & \blacksquare & \blacksquare\\
      &       &       &        &\circ & \circ & \circ\\
      &       &       &       &        &\circ & \circ\\
      &       &       &       &       &        &\circ\\
}
\]

In the above example we have: $c_1=4$, $c_2=2$, $c_3=1$, $c_4=0$, $r_1=3$, $r_2=2$, $r_3=1$ and $r_4=1$; these numbers can be easily deduced by the picture counting the numbers of black squares in $A$ for each row and each column in $\nnabla$.
\end{example}

Our next aim is to find a formula for $\tau_A^{-1}$, when restricted to the elements of $\nnabla$.

\begin{proposition}\label{proposition_tauInverseFormula}
Let $A$ be a $\nnabla$--ideal. We have the following formula for the restriction of $\tau_A^{-1}$ to $\nnabla$:
\[
\nnabla\ni\alpha=\talpha_{n-k+1,n+h-1}\,\longmapsto\,
\left\{
\begin{array}{ll}
-\talpha_{n+h-r_h,n-k+c_k} & \textrm{if }\alpha\in A,\\
\talpha_{n-k+c_k+1,n+h-r_h-1} & \textrm{if }\alpha\not\in A.\\
\end{array}
\right.
\]
\end{proposition}
\begin{proof}
We apply induction to $|A|$. If $A$ is empty, then $\tau_A=\id$, $c_1=c_2=\cdots = c_n=0$, $r_1=r_2=\cdots=r_n=0$ and the claimed formula for $\tau_A^{-1}$ reads just $\alpha\longmapsto\alpha$ since $\alpha\not\in A$ for all $\alpha\in\nnabla$.

Now suppose $|A|\geq 1$, let $\beta=\talpha_{n-v+1,n+u-1}$ be a $\myleq$--maximal element in $A$ and set $A'=A\setminus\{\beta\}$. Then $A'$ is a $\nnabla$--ideal and the claimed formula holds for $A'$ by induction; note that the $\myleq$--maximality of $\beta$ implies that $\tau_A^{-1}=\tau_{A'}^{-1}s_\beta$. We have also $c_v=u$ and $r_u=v$. Moreover for $k\neq v$, $c'_k:=c_k(A')$ is equal to $c_k$, and $c'_v = c_v-1=u-1$; analogously, for $h\neq u$, $r'_h=r_h(A')$ is equal to $r_h$, and $r'_u = r_u - 1=v-1$.

Now let $\alpha=\talpha_{n-k+1,n+h-1}$ be an element of $\nnabla$. According to Lemma \ref{lemma_nablaKey}, we consider three different cases as follows
\[
s_\beta(\alpha) =
\left\{
\begin{array}{lll}
\alpha & \textrm{if }k\neq v,\,h\neq u & (\text{Case 1})\\
\alpha-\beta & \textrm{if }(k=v,\, h\neq u) \textrm{ or }(k\neq v,\,h=u) & (\text{Case 2})\\
-\alpha & \textrm{if }k=v,\,h=u & (\text{Case 3})\\
\end{array}
\right.
\]
\begin{enumerate}
\item[(Case 1)] In this case $\alpha\neq\beta$. The formula holds by noticing that $\tau_A^{-1}(\alpha)=\tau_{A'}^{-1}s_\beta(\alpha)=\tau_A'^{-1}(\beta)$, $c_k=c_k'$ and $r_h=r_h'$.

\item[(Case 2)] We suppose $k=v$ and $h\neq u$, the other case is analogous. First assume that $\alpha\in A$, then $\tau_A^{-1}(\alpha)=\tau_{A'}^{-1}(\alpha-\beta)=\tau_{A'}^{-1}(\alpha) - \tau_{A'}^{-1}(\beta)$. We compute these last two terms.

For the first one, note that $\alpha\neq\beta$ implies $\alpha\in A'$. Then we have
\[
\tau_{A'}^{-1}(\alpha) = -\talpha_{n+h-r'_h,n-k+c'_k} = - \talpha_{n+h-r_h,n-v+u-1}
\]
where the first equality holds by induction, and for the second one we have used $r'_h=r_h$ and $c'_k=c'_v=u-1$.

For the second term, using $\beta\not\in A'$, we find
\[
\tau_{A'}^{-1}(\beta) = \talpha_{n-v+c'_v+1,n+u-r'_u-1} = \talpha_{n-v+u,n+u-v}
\]
where the first equality holds by induction, and for the second equality we have used $c'_v=u-1$ and $r'_u=v-1$.

Putting them together we conclude that 
$$\tau_A^{-1}(\alpha) = - \talpha_{n+h-r_h,n-v+u-1} - \talpha_{n-v+u,n+u-v} = -\talpha_{n+h-r_h,n-v+u},$$ 
and the claimed formula holds since $k=v$ and $c_k = u$.

It remains to consider the case $\alpha\not\in A$. Then $\alpha\not\in A'$ and, as above, $\tau_A^{-1}(\alpha)=\tau_{A'}^{-1}(\alpha) - \tau_{A'}^{-1}(\beta)$. Moreover, by induction and by $c'_v=u-1$ and $r'_h=r_h$, we have
\[
\tau_{A'}^{-1}(\alpha) = \talpha_{n-k+c'_k+1,n+h-r'_h-1} = \talpha_{n-v+u,n+h-r_h-1}.
\]
The above proved formula $\tau_{A'}^{-1}(\beta)=\talpha_{n-v+u,n+u-v}$ is still valid. So we conclude that 
$$\tau_A^{-1}(\alpha) = \talpha_{n-v+u,n+h-r_h-1} - \talpha_{n-v+u,n+u-v} = \talpha_{n-v+u+1,n+h-r_h-1}$$ 
and the claimed formula holds since $k=v$ and $c_v=u$.

\item[(Case 3)] In this case we have $\alpha=\beta$, hence $\tau_A^{-1}(\alpha)=-\tau_{A'}^{-1}(\beta)=-\talpha_{n-v+u,n+u-v}$ where the last equality follows by the computation executed above. Hence the claimed formula holds since $r_h=v$ and $c_k=u$.
\end{enumerate}
\end{proof}

\begin{example}\label{Ex:running4}
We continue Example \ref{Ex:running3}. In this case, we find
\[
\xymatrix@!C=5pt@!R=5pt{
\talpha_{1} \\
-\talpha_{2}&\ \  \talpha_{3} \\
-\talpha_{2,4}&\ \  -\talpha_{4}&\ \  \talpha_{5} \\
-\talpha_{2,7}&\ \  -\talpha_{4,7}&\ \  -\talpha_{6,7}&\ \  -\talpha_{7}\\
}
\]
where we are reported the image via $\tau_A^{-1}$ of the elements in $\nnabla$ arranged as in the previous picture.
\end{example}

As a first consequence of the previous formula we have the following proposition.
\begin{proposition}\label{proposition_NtauInverse}
If $A$ is a $\nnabla$--ideal then $\wt{N}(\tau_A^{-1})=A$. In particular the length of $\tau_A$ is $|A|$.
\end{proposition}
\begin{proof}
By Proposition \ref{proposition_tauInverseFormula}, $A\subseteq \wt{N}(\tau_A^{-1})$ and also $\nnabla\setminus A \subseteq \wt{P}(\tau_A^{-1})$. Since $\tPhi^+=\wt{N}(\tau_A^{-1})\sqcup\wt{P}(\tau_A^{-1})$, it suffices to show that for any $\alpha=\talpha_{i,j}\in\tPhi^+\setminus\nnabla$, $\tau_A^{-1}(\alpha)$ is a positive root. There are three cases to consider:

\begin{enumerate}
\item[(1)] Assume that $j<n$, then $\alpha=\talpha_{i,n}-\talpha_{j+1,n}$. We set $k=n-i+1$ and $k'=n-j$.
\begin{itemize}
\item If $\talpha_{i,n}\in A$, as $A$ is a $\nnabla$--ideal, $\talpha_{j,n}\in A$. By Proposition \ref{proposition_tauInverseFormula}, we have 
$$\tau_A^{-1}(\alpha) = -\talpha_{n+1-r_1,i-1+c_k}+\talpha_{n+1-r_1,j+c_{k'}}.$$ 
Note that $i\leq j$ implies $i<j+1$, hence $k>k'$ and $c_k\leq c_{k'}$; we conclude $j+c_{k'}\geq i-1+c_k$. This proves that $\tau_A^{-1}(\alpha)$ is a positive root.

\item If $\talpha_{i,n}\not\in A$ and $\talpha_{j+1,n}\in A$ then by Proposition \ref{proposition_tauInverseFormula}, $\tau_A^{-1}(\talpha_{i,n})\in\tPhi^+$ and $\tau_A^{-1}(\talpha_{j+1,n})\in\tPhi^-$, so $\tau_A^{-1}(\alpha)$ is positive.

\item If $\talpha_{i,n}\not\in A$ and $\talpha_{j+1,n}\not\in A$, by Proposition \ref{proposition_tauInverseFormula}, 
$$\tau_A^{-1}(\alpha) = \talpha_{i+c_k,n-r_1}-\talpha_{j+1+c_{k'},n-r_1}$$ 
and we conclude by $i+c_k\leq j+1+c_{k'}$ as already proved.
\end{itemize}

\item[(2)] When $i>n$, we proceed in a similar way as in (1) by symmetry.

\item[(3)] Assume that $j-i\geq n$: in this case $i<n<j$.

If $\talpha_{i,n}\not\in A$ then $\tau_A^{-1}(\talpha_{i,n})\in\tPhi^+$ since $\talpha_{i,n}\in\nnabla\setminus A$; moreover by the case (2) above, $\tau_A^{-1}(\talpha_{n+1,j})\in\tPhi^+$. Writing $\alpha=\talpha_{i,n}+\talpha_{n+1,j}$, the above argument shows that $\tau_A^{-1}(\alpha)\in\tPhi^+$. Similarly, if $\talpha_{n,j}\not\in A$ we argue in the same way.

It remains to consider the case $\talpha_{i,n},\talpha_{n,j}\in A$. We write $\alpha=\talpha_{i,n}-\talpha_{n,n}+\talpha_{n,j}$ and note that the image of these three roots are $-\talpha_{n+1-r_1,i-1+c_k}$, $\talpha_{n+1-r_1,n-1+c_1}$ and $-\talpha_{j+1-r_h,n-1+c_1}$ respectively, where $h=j+1-n$ and $k=n-i+1$. By the assumption $j-i\geq n$, it is easy to see $r_h<k$ and $c_k<h$, hence $i-1+c_k<j+1-r_h$ and this proves that $\tau_A^{-1}(\alpha)\in\tPhi^+$.
\end{enumerate}
\end{proof}

Another consequence of Proposition \ref{proposition_tauInverseFormula} is the following property related to the Bruhat order of $\tPhi$.

\begin{proposition}\label{proposition_tauAntiIsomorphism}
The map
\[
\nnabla\supseteq A \ni\alpha\,\stackrel{i_A}{\longmapsto}\,-\tau_A^{-1}(\alpha) \in \wt{N}(\tau_A) \subseteq \tPhi^+
\]
is a poset anti-isomorphism. Moreover for $\alpha,\beta\in A$ we have:
\begin{itemize}
	\item[(i)] if $\alpha\meet\beta$ exists then it is in $A$ and $i_A(\alpha\meet\beta) = i_A(\alpha) \join i_A(\beta)$; on the other hand if $i_A(\alpha)\join i_A(\beta)$ is in $\wt{N}(\tau_A)$ then $\alpha\meet\beta$ exists;
	\item[(ii)] if $\alpha\join\beta$ is in $A$ then $i_A(\alpha)\meet i_A(\beta)$ exists and it is equal to $i_A(\alpha \join \beta)$; on the other hand if $i_A(\alpha)\meet i_A(\beta)$ exists then $\alpha\join\beta$ is in $A$.
\end{itemize}
\end{proposition}
\begin{proof}
By Proposition \ref{proposition_NtauInverse}, $A=N(\tau_A^{-1})$, hence $i_A(A)\subseteq\tPhi^+$ and $i_A$ maps $A$ to $\wt{N}(\tau_A)\subseteq\tPhi^+$.

Let $\alpha=\talpha_{n-k+1,n+h-1}$ and $\beta=\talpha_{n-v+1,n+u-1}$ be two elements of $A$. By Proposition \ref{proposition_tauInverseFormula}, $i_A(\alpha) = \talpha_{n+h-r_h,n-k+c_k}$ and $i_A(\beta)=\alpha_{n+u-r_u,n-v+c_v}$. Note that $\alpha\leq\beta$ if and only if $k\leq v$ and $h\leq u$. Recalling that $t\longmapsto r_t$ and $t\longmapsto c_t$ are both non-increasing maps, we easily see that $k\leq v$ and $h\leq u$ if and only if $n+u-r_u\geq n+h-r_h$ and $n-v+c_v\leq n-k+c_k$; this last condition is clearly equivalent to $i_A(\alpha)\geq i_A(\beta)$. This finishes the proof that $i_A$ is a poset anti-isomorphism.

Now we prove (i). If $\alpha\meet\beta$ exists then it is an element of $A$ since $A$ is a $\nnabla$--ideal and $\alpha,\beta\geq\talpha_{n,n}$. Let $a=\min(k,v)$ and $b=\min(h,u)$. Then $\alpha\meet\beta = \talpha_{n-a+1,n+b-1}$. By Proposition \ref{proposition_tauInverseFormula}, $i_A(\alpha\meet\beta)=\talpha_{n+b-r_b,n-a+c_a}$. Since $t\longmapsto r_t,c_t$ are both non-increasing, we find $\min(n+h-r_h,n+u-r_u)=n+b-r_b$ and $\max(n-k+c_k,n-v+c_v)=n-a+c_a$, this proves that $i_A(\alpha)\join i_A(\beta)=i_A(\alpha\meet\beta)$.

On the other hand note that $\eta = i_A(\alpha)\join i_A(\beta)$ always exists. If it is an element of $\wt{N}(\tau_A)$, then there exists $\gamma\in A$ such that $i_A(\gamma)=\eta$ since $i_A:A\longrightarrow \wt{N}(\tau_A)$ is bijective. Since $i_A$ an anti-isomorphism of posets and $\eta\geq i_A(\alpha),i_A(\beta)$, we get $\alpha,\beta\geq\gamma$, and hence $\alpha\meet\beta$ exists.

The proof of (ii) is analogous, so omitted.
\end{proof}

\begin{example}\label{Ex:running5}
We continue Example \ref{Ex:running4}. Let $\tau_A=s_4s_5s_6s_7s_3s_4s_2$. Then 
$$\wt{N}(\tau_A)=\{\talpha_2,\talpha_{2,4},\talpha_{2,7},\talpha_{4},\talpha_{4,7},\talpha_{6,7},\talpha_7\}.$$
In this example, the map $i_A$ is given by:
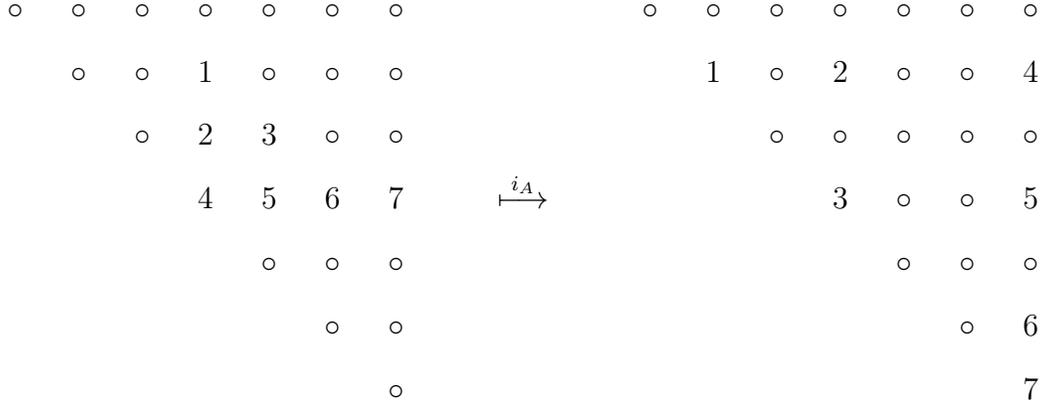
\[
\xymatrix@!C=0pt@!R=0pt{
 \circ & \circ & \circ & \circ & \circ & \circ & \circ & &                             & & \circ & \circ & \circ & \circ & \circ & \circ & \circ\\
       & \circ & \circ & 1     & \circ & \circ & \circ & &                             & &       & 1     & \circ & 2     & \circ & \circ & 4\\
       &       & \circ & 2     & 3     & \circ & \circ & &                             & &       &       & \circ & \circ & \circ & \circ & \circ\\
       &       &       & 4     & 5     & 6     & 7     & & \stackrel{i_A}{\longmapsto} & &       &       &       & 3     & \circ & \circ & 5\\
       &       &       &       & \circ & \circ & \circ & &                             & &       &       &       &       & \circ & \circ & \circ\\
       &       &       &       &       & \circ & \circ & &                             & &       &       &       &       &       & \circ & 6\\
       &       &       &       &       &       & \circ & &                             & &       &       &       &       &       &       & 7\\
}
\]
Our convention is that the positive root at position $k=1,2,\cdots,7$ is mapped by $i_A$ to the positive root labelled by the same number in the picture on the right hand side.

\end{example}

The following proposition on representable functions on $\nnabla$--ideals will be used to construct a suitable weight associated with a pair $(\lambda,\tau)$ with $\tau$ an element of $W$ and $\lambda$ a dominant weight for $\Phi$.

\begin{proposition}\label{proposition_nablaIdealWeight}
Let $A$ be a $\nnabla$--ideal and let $f:A\longrightarrow\Z$ be a map such that:
\begin{itemize}
	\item[(i)] for any $\alpha \in A$ we have $f(\alpha) \leq 0$,
	\item[(ii)] if $\alpha,\beta\in A$ and $\alpha\leq\beta$ then $f(\alpha) \leq f(\beta)$,
	\item[(iii)] if $\alpha,\beta,\alpha\join\beta\in A$ then $f(\alpha\join\beta)=f(\alpha) + f(\beta) - f(\alpha\meet\beta)$,
	\item[(iv)] if $\alpha,\beta\in A$ and $f(\alpha)+f(\beta)-f(\alpha\meet\beta)<0$ then $\alpha\join\beta\in A$. 
\end{itemize}
Then there exists $\mu\in\tLambda$ such that $\pair{\mu}{\alpha} = f(\alpha)$ for all $\alpha\in A$ and $\pair{\mu}{\alpha}\geq 0$ for all $\alpha\in\tPhi^+\setminus A$. In particular $\tau_A^{-1}\mu$ is a dominant weight.
\end{proposition}

\begin{proof}
If $A=\varnothing$ then we may take $\mu=0$. From now on we assume that $A$ is non-empty. Being a $\nnabla$--ideal, we have $\talpha_{n,n}\in A$.

We define certain integers $c_1,c_2,\ldots,c_{2n-1}$ we use later in a formula for $\mu$ as in the statement of the Proposition. First we set $c_n := f(\talpha_{n,n})$. Let $h$ be the minimum such that $\alpha_{h,n}\in A$ and $k$ be the maximum such that $\alpha_{n,k}\in A$. Since the subset $A$ is a $\nnabla$--ideal, we have:
\begin{itemize}
\item $\talpha_{t,s}\not\in A$ for all $1\leq t\leq h-1$ and all $t\leq s\leq 2n-1$,
\item $\talpha_{t,n}\in A$ for all $h\leq t\leq n$,
\item $\talpha_{t,s}\not\in A$ for all $k+1\leq s\leq 2n-1$ and all $1\leq t\leq s$,
\item $\talpha_{n,s}\in A$ for all $n\leq s\leq k$.
\end{itemize}
\noindent We set $c_t = f(\talpha_{t,n}) - f(\talpha_{t+1,n})$ for all $t=h,h+1,\ldots,n-1$ and $c_t=f(\talpha_{n,t})-f(\talpha_{n,t-1})$ for $t=n+1,n+2,\ldots,k$. Moreover let $c_{h-1} = -f(\talpha_{h,n})$, $c_{k+1} = -f(\talpha_{n,k})$ and $c_t=0$ for all $1\leq t\leq h-2$ and all $k+2\leq t\leq 2n-1$.

We will show that 
$$\mu:=\sum_{t=1}^{2n-1}c_t\wt{\varpi}_t\in\tLambda$$ 
verifies the properties requested in the statement of the proposition. 

We first show that for any $\alpha\in A$, $\pair{\mu}{\alpha}=f(\alpha)$. Clearly $\langle\mu,\talpha^\vee_{n,n}\rangle = c_n =f(\talpha_{n,n})$. For $h\leq i\leq n-1$ we have
\[
\langle\mu,\talpha^\vee_{i,n}\rangle = c_i + c_{i+1} + \cdots + c_{n-1} + c_n = f(\talpha_{i,n})
\]
and, in the same way, for all $n+1\leq j\leq k$,
$$\langle\mu,\talpha^\vee_{n,j}\rangle=c_n + c_{n+1} + \cdots + c_j = f(\talpha_{n,j}).$$

Now let $\talpha_{i,j}$ be such that $i\leq n\leq j$ (or, equivalently, such that $\talpha_{n,n}\leq\talpha_{i,j}$). Then $\talpha_{i,n}\join\talpha_{n,j}=\talpha_{i,j}$, $\talpha_{i,n}\meet\talpha_{j,n}=\talpha_{n,n}$ and $\talpha_{i,j}=\talpha_{i,n}+\talpha_{n,j}-\talpha_{n,n}$.

Since $A$ is a $\nnabla$--ideal, $\talpha_{i,j}\in A$ implies $\talpha_{i,n},\talpha_{n,j}\in A$. By (iii), 
$$f(\talpha_{i,n}\join\talpha_{n,j})=f(\talpha_{i,n})+f(\talpha_{n,j})-f(\talpha_{n,n})$$
and, using what we have already proved, we get
\[
\langle\mu,\talpha^\vee_{i,j}\rangle = \langle\mu,\talpha^\vee_{i,n}\rangle + \langle\mu,\talpha^\vee_{n,j}\rangle - \langle\mu,\talpha^\vee_{n,n}\rangle = f(\talpha_{i,n}) + f(\talpha_{n,j}) - f(\talpha_{n,n}) = f(\talpha_{i,j}).
\]

We show that $\langle\mu,\talpha^\vee_{i,j}\rangle\geq 0$ for all $\talpha_{i,j}\not\in A$. Note that, by (i) and (ii), $c_t\geq 0$ for all $h\leq t\leq k$ but $c_n$ is negative. It is then clear that $\langle\mu,\talpha^\vee_{i,j}\rangle\geq 0$ if $j<n$ or $n<i$. 

We assume that $i\leq n\leq j$. If furthermore $h\leq i$ and $j\leq k$ then, as above, 
$$\langle\mu,\talpha^\vee_{i,j}\rangle = f(\talpha_{i,n}) + f(\talpha_{n,j}) - f(\talpha_{n,n})$$ 
and this is a non negative integer by (iv) since we are assuming $\talpha_{i,j}\not\in A$. On the other hand, if $i < h$ then 
$$\langle\mu,\talpha^\vee_{i,j}\rangle \geq  c_{h-1} + c_h + \cdots + c_n\geq c_{h-1} + c_{n} = c_{h-1}+f(\talpha_{h,n})=0$$ 
and, in the same way $\langle\mu,\talpha^\vee_{i,j}\rangle\geq 0$ if $j > k$. This finishes the proof of the required properties for $\mu$.

It remains to show that $\tau_A^{-1}\mu$ is a dominant weight. This is equivalent to $\pair{\tau_A^{-1}\mu}{\beta}\geq 0$ for all $\beta\in\tPhi^+$. Fix such a $\beta$ and let $\alpha:=-\tau_A\beta$. If $\alpha$ is a positive root then $\alpha\in \wt{N}(\tau_A^{-1})$, so $\alpha\in A$ by Proposition \ref{proposition_NtauInverse}. Hence $\pair{\tau_A^{-1}\mu}{\beta}=-\pair{\mu}{\alpha}=-f(\alpha)\geq 0$. On the other hand, if $\alpha$ is a negative root, then $-\alpha\in \wt{P}(\tau_A^{-1})$, so $-\alpha\in\tPhi^+\setminus N(\tau_A^{-1})=\tPhi^+\setminus A$, again by Proposition \ref{proposition_NtauInverse}; hence $\pair{\tau_A^{-1}\mu}{\beta}=\pair{\mu}{-\alpha}\geq 0$ by what have already proved for $\mu$.
\end{proof}

\begin{remark}
The weight $\mu$, such that the associated function $f$ fulfils the properties (i)-(iv) in the proposition, is in general not unique.
\end{remark}

%\begin{remark}
%The condition (iii) and (iv) in Proposition \ref{proposition_nablaIdealWeight} are equivalent to the same conditions with $\alpha = \talpha_{i,n}$ and $\beta = \talpha_{n,j}$ (and then $\alpha\meet\beta = \talpha_n$) as one can prove easily. Indeed condition (iii) has been used only in this form in the proof of the previous proposition.
%\end{remark}

\subsection{From rectangular elements of $\SL_{n+1}$ to $\SL_{2n}$}\label{Sec:D}

Let $\tau\in W$ be an irreducible rectangular element. Recall that in Section \ref{section_rectangularElements} for a subset $A\subseteq\Phi^+$ we have defined $A^-=A\cap\{\alpha_{1,1},\alpha_{1,2},\ldots,\alpha_{1,n}\}$ and $A^+=A\cap\{\alpha_{1,n},\alpha_{2,n},\ldots,\alpha_{n,n}\}$. In particular let $n = j_1 > j_2 > \cdots > j_s \geq 1$ and $1=i_1 < i_2 < \cdots < i_r\leq n$ be such that $N(\tau)^{-}=\{\alpha_{1,j_1},\alpha_{1,j_2},\ldots,\alpha_{1,j_s}\}$ and $N(\tau)^+=\{\alpha_{i_1,n},\alpha_{i_2,n},\ldots,\alpha_{i_r,n}\}$ (note that, $\alpha_{1,j_1} = \alpha_{i_1,n} = \theta$). By Proposition \ref{proposition_irreducibleRectangular}, $N(\tau)$ is the set of all $\alpha_{i_h,j_k}$ with $1\leq h\leq r$, $1\leq k\leq s$ and $i_h\leq j_k$. We define a map $\mathcal{D}:N(\tau)\ra\nnabla$ by:
\[
N(\tau)\ni \alpha_{i_h,j_k} \stackrel{\mathcal{D}}{\longmapsto} \talpha_{n-k+1,n+h-1} \in \nnabla\subseteq\tPhi^+.
\]
Let $A\subseteq\nnabla$ be the image of $\mathcal{D}$. It is clear that $\mathcal{D}$ is a bijection from $N(\tau)$ to $A$.

\begin{lemma}\label{lemma_imageDIdeal}
If $\tau$ is an irreducible rectangular element then the image $A = \mathcal{D}(N(\tau))$ is a $\nnabla$--ideal in $\tPhi^+$.
\end{lemma}
\begin{proof}
Let $\talpha_{n-k+1,n+h-1}=\mathcal{D}(\alpha_{i_h,j_k})\in A$, for some $1\leq h\leq r$ and $1\leq k\leq s$, and let $\talpha_{n-u,n+v}$, for some $0\leq u,v\leq n-1$, be an element of $\nnabla$ with $\talpha_{n-u,n+v}\leq\talpha_{n-k+1,n+h-1}$. Then $u+1\leq k$, $v+1\leq h$, hence $i_{v+1}\leq i_h\leq j_k\leq j_{u+1}$, where the second inequality follows by $\alpha_{i_h,j_k}\in N(\tau)$. We have thus proved $i_{v+1}\leq j_{u+1}$, so $\alpha_{i_{v+1},j_{u+1}}$ is an element of $N(\tau)$ and clearly $\mathcal{D}(\alpha_{i_{v+1},j_{u+1}})=\talpha_{n-u,n+v}$. We conclude that $\talpha_{n-u,n+v}$ is an element of $A$. This finishes the proof that $A$ is a $\nnabla$--ideal.
\end{proof}

\begin{proposition}\label{proposition_DantiIsomorphism}
If $\tau$ is an irreducible rectangular element then the map $\mathcal{D}:N(\tau)\longrightarrow A$ is an anti-isomorphism of posets. Moreover for $\alpha,\beta\in N(\tau)$ we have:
\begin{itemize}
	\item[(i)] if $\alpha\meet\beta$ exists then it is in $N(\tau)$ and $\mathcal{D}(\alpha\meet\beta) = \mathcal{D}(\alpha)\join \mathcal{D}(\beta)$; on the other hand if $\mathcal{D}(\alpha)\join \mathcal{D}(\beta)$ is in $A$ then $\alpha\meet\beta$ exists;
	\item[(ii)] $\alpha\join\beta$ is in $N(\tau)$, $\mathcal{D}(\alpha)\meet \mathcal{D}(\beta)$ exists and is an element of $A$, moreover $\mathcal{D}(\alpha\join\beta)=\mathcal{D}(\alpha)\meet \mathcal{D}(\beta)$.
\end{itemize}
\end{proposition}
\begin{proof} 
First notice that the following statements are equivalent:
\begin{enumerate}
\item $\alpha = \alpha_{i_h,j_k}\leq\alpha_{i_u,j_v}=\beta$;
\item $i_u\leq i_h$ and $j_v\geq j_k$;
\item $u\leq h$ and $v\leq k$;
\item $\mathcal{D}(\alpha)=\talpha_{n-k+1,n+h-1}\geq\talpha_{n-v+1,n+u-1}=\mathcal{D}(\beta)$.
\end{enumerate}

The equivalence between (1) and (4) shows that $\mathcal{D}$ is a poset anti-isomorphism.

Now we prove (i). If $\alpha\meet\beta$ exists in $\Phi^+$ then it is in $N(\tau)$ since $\tau$ is (irreducible) rectangular and the formula $\mathcal{D}(\alpha\meet\beta)=\mathcal{D}(\alpha)\join \mathcal{D}(\beta)$ follows being $\mathcal{D}$ an anti-isomorphism of posets. On the other hand if $\eta=\mathcal{D}(\alpha)\join \mathcal{D}(\beta)$ is in $A$, then $\eta=\mathcal{D}(\gamma)$ for certain $\gamma\in N(\tau)$, so $\gamma\leq\alpha,\beta$ by $\eta\geq \mathcal{D}(\alpha),\mathcal{D}(\beta)$. This shows that $\alpha\meet\beta$ exists.

For (ii) note that $\alpha\join\beta$ is in $N(\tau)$ by Proposition \ref{proposition_irreducibleRectangular} while $\mathcal{D}(\alpha)\meet \mathcal{D}(\beta)$ exists and is an element of $A$ since $A$ is a $\nnabla$--ideal by Lemma \ref{lemma_imageDIdeal}. Again, the formula $\mathcal{D}(\alpha\join\beta)=\mathcal{D}(\alpha)\meet \mathcal{D}(\beta)$ follows being $\mathcal{D}$ an anti-isomorphism of posets.
\end{proof}

\begin{example}\label{Ex:running6}
Let $n=4$. We consider $\tau=s_1s_2s_3s_4s_1s_2s_1\in W$ as in Example \ref{Ex:running1}. The map $\mathcal{D}$ sending $N(\tau)$ to $\nnabla$ is depicted as follows:
\[
\xymatrix@!C=0pt@!R=0pt{
 1    & 2 & \circ & 4  & &                                     & & \circ & \circ & \circ & \circ & \circ & \circ & \circ\\
      & 3 & \circ & 5  & &                                     & &       & \circ & \circ & 1     & \circ & \circ & \circ\\
      &   & \circ & 6  & & \stackrel{\mathcal{D}}{\longmapsto} & &       &       & \circ & 2     & 3     & \circ & \circ\\
      &   &       & 7  & &                                     & &       &       &       & 4     & 5     & 6     & 7\\
      &   &       &    & &                                     & &       &       &       &       &\circ  & \circ & \circ\\
      &   &       &    & &                                     & &       &       &       &       &       & \circ & \circ\\
      &   &       &    & &                                     & &       &       &       &       &       &       &\circ\\
}
\]
Here we used the same convention as in Example \ref{Ex:running5}.
\end{example}

In Section \ref{section_nablaIdeal} we have associated an element $\ttau:= \tau_A$ in the Weyl group $\tW$ of $\tPhi$ to the $\nnabla$--ideal $A$. Combining the previous proposition with Corollary \ref{proposition_tauAntiIsomorphism} gives:

\begin{corollary}\label{corollary_latticeIsomorphism}
Let $\tau$ be an irreducible rectangular element. The composition
\[
\Phi^+\supseteq N(\tau)\stackrel{\mathcal{D}}{\longrightarrow} A \stackrel{i_A}{\longrightarrow} \wt{N}(\ttau)\subseteq\tPhi^+
\]
is an isomorphism of posets, it preserves the join and, when defined the meet. Moreover $\ttau$ is an irreducible rectangular element of the parabolic subgroup $\tW_{\tI}$ of $\tW$ for a, in general proper, connected subset $\tI$ of $\tDelta$.
\end{corollary}
\begin{proof}
By Proposition \ref{proposition_tauAntiIsomorphism} and \ref{proposition_DantiIsomorphism}, this composition is a poset isomorphism. Being $i_A$ an anti-isomorphism, the union $\tI$ of the supports of the roots in $\wt{N}(\ttau)$ is the support of $i_A(\talpha_{n,n})$, hence $\tI$ is connected. Now it is clear that $\wt{N}(\ttau)$ is irreducible rectangular, with respect to $\tI$, by Proposition \ref{proposition_irreducibleRectangular} using the properties of $\mathcal{D}$ and $i_A$ about meet and join stated in Proposition \ref{proposition_tauAntiIsomorphism} and in Proposition \ref{proposition_DantiIsomorphism}.
\end{proof}

\begin{example}
Combining Example \ref{Ex:running5} and \ref{Ex:running6} we have the following picture for $i_A\circ\mathcal{D}$:
\[
\xymatrix@!C=0pt@!R=0pt{
 1    & 2 & \circ & 4 & &             & & \circ & \circ & \circ & \circ & \circ & \circ & \circ\\
      & 3 & \circ & 5 & &             & &       & 1     & \circ & 2     & \circ & \circ & 4\\
      &   & \circ & 6 & & \longmapsto & &       &       & \circ & \circ & \circ & \circ & \circ\\
      &   &       & 7 & &             & &       &       &       & 3     & \circ & \circ & 5\\
      &   &       &   & &             & &       &       &       &       & \circ & \circ & \circ\\
      &   &       &   & &             & &       &       &       &       &       & \circ & 6\\
      &   &       &   & &             & &       &       &       &       &       &       & 7\\
}
% \longmapsto
% \xymatrix@!C=0pt@!R=0pt{
%  \circ & \circ & \circ & \circ & \circ & \circ & \circ\\
%       &1 & \circ & 2 & \circ & \circ & 4\\
%       &        &\circ & \circ & \circ & \circ & \circ\\
%       &       &        &3 & \circ & \circ & 5\\
%       &       &       &        &\circ & \circ & \circ\\
%       &       &       &       &        &\circ & 6\\
%       &       &       &       &       &        &7\\
% }
\]

\end{example}

\begin{remark}
As we have seen from the above example, the poset $\wt{N}(\ttau)$ is isomorphic to $N(\tau)$. But the positive roots in $\wt{N}(\ttau)$ are \emph{commutative}, i.e. $\alpha+\beta$ is never a root for $\alpha,\beta\in\wt{N}(\ttau)$, which is not the case for $N(\tau)$.
\end{remark}

Let $\lambda\in\Lambda^+$ be a dominant weight. We define
\[
A\ni \mathcal{D}(\alpha) \stackrel{f}{\longmapsto} -\pair{\lambda}{\alpha}\in\Z.
\]
\begin{lemma}
The subset $A$ and the map $f:A\longrightarrow\Z$ fulfil the hypothesis of Proposition \ref{proposition_nablaIdealWeight}.
\end{lemma}
\begin{proof} By Lemma \ref{lemma_imageDIdeal} $A$ is a $\nnabla$--ideal and, being $\lambda$ dominant, it is clear that $f(\mathcal{D}(\alpha))\leq 0$ for all $\alpha \in N(\tau)$ and (i) in Proposition \ref{proposition_nablaIdealWeight} holds.

If $\mathcal{D}(\alpha)\leq \mathcal{D}(\beta)$ are elements of $A$ then by Proposition \ref{proposition_DantiIsomorphism} we have $\alpha\geq\beta$. So, being $\lambda$ dominant, $f(\mathcal{D}(\alpha))\leq f(\mathcal{D}(\beta))$ and also (ii) in Proposition \ref{proposition_nablaIdealWeight} holds.

Now we prove that the hypothesis (iii) is fulfilled. Let $\alpha,\beta\in N(\tau)$ and $\talpha=\mathcal{D}(\alpha)$, $\tbeta=\mathcal{D}(\beta)$. If $\talpha\join\tbeta\in A$ then $\alpha\meet\beta\in N(\tau)$ by (i) of Proposition \ref{proposition_DantiIsomorphism} and $\mathcal{D}(\alpha\meet\beta)=\talpha\join\tbeta$. Moreover $\talpha\meet\tbeta\in A$, being $A$ a $\nnabla$--ideal, $\alpha\join\beta\in N(\tau)$, being $\tau$ irreducible rectangular, and $\mathcal{D}(\alpha\join\beta) = \talpha\meet\tbeta$ by (ii) of Proposition \ref{proposition_DantiIsomorphism}. Finally since $\alpha\meet\beta + \alpha\join\beta=\alpha+\beta$, we find
\[
\begin{array}{rcl}
f(\talpha\join\tbeta) & = & -\pair{\lambda}{(\alpha\meet\beta)}\\
 & = & -\pair{\lambda}{(\alpha+\beta-\alpha\join\beta)}\\
 & = & -\pair{\lambda}{\alpha}-\pair{\lambda}{\beta}+\pair{\lambda}{(\alpha\join\beta)}\\
 & = & f(\talpha) + f(\tbeta) - f(\talpha\meet\tbeta),
\end{array}
\]
and (iii) is proved.

Finally, in order to prove that (iv) holds, note that, using the same notations as above, 
$$f(\talpha)+f(\tbeta)-f(\talpha\meet\tbeta)=-\pair{\lambda}{(\alpha+\beta-(\alpha\join\beta))}.$$ 
Since the support of $\alpha\join\beta$ is the smallest convex subset containing $\supp(\alpha)\cup\supp(\beta)$, we see that if this integer is negative then the intersection of $\supp(\alpha)$ and $\supp(\beta)$ is not empty. But $\tau$ is rectangular, hence this implies that $\alpha\meet\beta$ exists in $\Phi^+$; hence it is in $N(\tau)$ by being $\tau$ rectangular and so $\talpha\join\tbeta=\mathcal{D}(\alpha\meet\beta)$ is an element of $A$.
\end{proof}

We can now apply Proposition \ref{proposition_nablaIdealWeight} and conclude as in the following corollary.
\begin{corollary}\label{corollary_finalResultCombinatorics}
Let $\tau$ be an irreducible rectangular element of $W$ and $\lambda$ be a dominant weight in $\Lambda^+$. Then there exists a dominant weight $\tlambda\in\tLambda^+$ and an element $\ttau$ of the Weyl group of $\tPhi$ such that: $\pair{\ttau\tlambda}{\talpha} = -\pair{\lambda}{\alpha}$ for $\talpha\in A$, with $\talpha=\mathcal{D}(\alpha)$ (so these scalar products are non-positive), and $\pair{\ttau\tlambda}{\beta} \geq 0$ if $\beta\in\tPhi^+\setminus A$.
\end{corollary}

A particular case will be important in our application to degenerate Schubert varieties.
\begin{proposition}\label{proposition_finalResultFundamental}
Let $\tau$ be an irreducible rectangular element of $W$ and let $\lambda = \varpi_r$ be a fundamental weight. Then, with notation as in the previous corollary, we can choose $\tlambda$ to be a fundamental weight $\wt{\varpi}_{r'}$. In particular $r'=n-v+u$ where $v=|\{\alpha_{1,j}\in N(\tau)\,|\,j\geq r\}|$ and $u=|\{\alpha_{i,n}\in N(\tau)\,|\,i\leq r\}|$.
\end{proposition}
\begin{proof}
We follow the construction of $\mu$ as in the proof of Proposition \ref{proposition_nablaIdealWeight}. First of all $\theta\in N(\tau)$, being $\tau$ irreducible rectangular, and $\mathcal{D}(\theta)=\talpha_{n,n}$ since $\mathcal{D}$ is an anti-isomorphism of posets; in particular $\pair{\mu}{\talpha_{n,n}}=-1$. Moreover $\pair{\mu}{\talpha}=0,-1$ for all $\talpha\in A$ by the definition using that $\lambda$ is fundamental.

Now let $\talpha_{i,j}\in A$. By $\talpha_{i,j}+\talpha_{n,n}=\talpha_{i,n}+\talpha_{n,j}$ we find: $\pair{\mu}{\talpha_{i,j}}=-1$ if and only if $\langle\mu,\talpha_{i,n}^\vee\rangle=\langle\mu,\talpha_{n,j}^\vee\rangle=-1$. It is then clear that $\mu = \wt{\varpi}_{n-v}-\wt{\varpi}_n+\wt{\varpi}_{n+u}$, with $u$ and $v$ as in the Proposition.

Hence $\pair{\mu}{\talpha}\geq -1$ for all $\talpha\in\tPhi^+$ and we conclude that $\tlambda$ is fundamental. Since $\tlambda$ is the unique dominant weight in the $\tW$--orbit of $\mu$ it is easy to check that $\tlambda = \wt{\varpi}_{n-v+u}$.
\end{proof}

\section{Proof of Theorem \ref{Thm:Schubert}}\label{Sec:Proof}

The proof of the theorem is executed in the following steps. We fix an irreducible rectangular element $\tau\in W$.

\subsection{Dimension estimation}

The first step of the proof is to show that the corresponding Demazure modules have the same dimension. The proof makes fully use of the basis constructed in \cite{FFL, Fou} and its parametrisation by the Feigin-Fourier-Littelmann-Vinberg (FFLV) polytope, as well as all results we have developed for rectangular elements.

\begin{proposition}\label{Prop:DimEst}
If $\tau\in W$ is an irreducible triangular element, then
\[
\operatorname{dim} V_\tau(\lambda) = \operatorname{dim} V_{\widetilde{\tau}}(\widetilde{\lambda}).
\]
\end{proposition}

\begin{proof}
Recall the marked chain polytope of the Gelfand-Tsetlin poset (for example \cite[Figure 1]{ABS}) associated to $\Phi^+$ and a marking vector $\lambda \in P^+$. By \cite[Proposition 2, Corollary 2]{Fou}, there exists a monomial basis of $V_\tau(\lambda)$ parametrised by the lattice points in the marked chain polytope associated to the subposet induced from $N(\tau) \subseteq \Phi^+$ and marking vector $\lambda$. \\
By Corollary~\ref{corollary_latticeIsomorphism}, $\ttau$ is a rectangular element hence a triangular element. With the same argument, there exists a monomial basis of $V_\ttau(\tlambda)$ parametrised by the lattice points in the marked chain polytope associated to the subposet induced from $\wt{N}(\ttau) \subseteq \widetilde{\Phi}^+$ and marking vector $\tlambda$. \\
Again by Corollary~\ref{corollary_latticeIsomorphism}, together with Corollary~\ref{corollary_finalResultCombinatorics} and Proposition~\ref{proposition_finalResultFundamental}, the map $i_A \circ \mathcal{D}$ together with the assignment $\lambda \longmapsto \tlambda$, defines an isomorphism of marked chain polytopes. \\
As conclusion, the two marked chain polytopes have the same number of lattice points, which proves the equality on the dimensions.
\end{proof}

\subsection{The case of fundamental weights}
Notice that the negative roots contained in $\mathfrak{n}_\tau^-$ are $-N(\tau)$. 

Since both Lie algebras $\mathfrak{n}_\tau^{-,a}$ and $\wt{\mathfrak{n}}_\ttau^+$ are abelian, the map $\mathcal{D}$ defined in the beginning of Section \ref{Sec:D} gives an isomorphism of Lie algebras
$$\Psi_\tau:\mathfrak{n}_\tau^{-,a}\stackrel{\sim}{\longrightarrow}\wt{\mathfrak{n}}_\ttau^+$$
sending the root space of weight $\alpha$ to that of weight $-\mathcal{D}(\alpha)$.

Then $\Psi_\tau$ induces an isomorphism of algebras
$$\Psi_{\tau}:S(\mathfrak{n}_\tau^{-,a})\stackrel{\sim}{\longrightarrow} S(\wt{\mathfrak{n}}_\ttau^+).$$

Let $\varpi_r$ be a fundamental weight of $G$. Recall that we have defined the cyclic module 
$$V_\tau^{-,a}(\varpi_r):=S(\mathfrak{n}_\tau^{-,a})\cdot v_{\varpi_r}^a \text{ (resp. }V^{a}(\varpi_r):=S(\mathfrak{n}^{-,a}\text)\cdot v_{\varpi_r}^a\text{)}.$$ 
Let $I_\tau^a(\varpi_r)$ (resp. $I^a(\varpi_r)$) denote its defining ideal (see Section \ref{Sec:PBW}) in $S(\mathfrak{n}_\tau^{-,a})$ (resp. $S(\mathfrak{n}^{-,a}$)). 

Let $\Phi^+_r$ denote the nilpotent radical consisting of all roots in $\Phi^+$ supported at $r$. We first study the defining ideal $I^a(\varpi_r)$.

\begin{lemma}\label{Lem:DefIdeal}
The defining ideal $I^a(\varpi_r)$ is generated by the following elements:
\begin{itemize}
\item[(G1)] $f_\alpha$ for $\alpha\in\Phi^+\setminus\Phi_r^+$;
\item[(G2)] $f_\alpha f_\beta+f_{\alpha\join\beta} f_{\alpha\meet\beta}$ for $\alpha,\beta\in\Phi_r^+$ such that both $\alpha\join\beta$ and $\alpha\meet\beta$ exist.
\end{itemize}
\end{lemma}

\begin{proof}
By \cite[Theorem 4]{FFL}, the defining ideal $I^a(\varpi_r)$ is given by:
$$I^a(\varpi_r)=S(\mathfrak{n}^{-,a})\cdot\mathrm{span}\left\{U(\mathfrak{n}^+)\circ f_\alpha^{\pair{\varpi_r}{\alpha}+1}\mid \alpha\in\Phi^+\right\},$$
where elements in $\mathfrak{n}^+$ act on $\mathfrak{n}^-$ as derivations by: for $\alpha,\beta\in\Phi^+$,
\[
e_\alpha\circ f_\beta =
\left\{
\begin{array}{ll}
f_{\beta-\alpha}, & \textrm{if }\beta-\alpha\in\Phi^+,\\
0, & \textrm{otherwise.}\\
\end{array}
\right.
\]

Notice that this action preserves the PBW degree. We compute $U(\mathfrak{n}^+)\circ f_\alpha^{\pair{\varpi_r}{\alpha}+1}$:
\begin{enumerate}
\item Let $\alpha\in\Phi^+\setminus\Phi^+_r$. Since the action of $U(\mathfrak{n}^+)$ preserves the PBW degree, $\{U(\mathfrak{n}^+)\circ f_\alpha\}$ is a linear span of elements of PBW degree one. This shows that
$$\{U(\mathfrak{n}^+)\circ f_\alpha\mid\alpha\in\Phi^+\setminus\Phi^+_r\}=\mathrm{span}\{f_\alpha\mid\alpha\in\Phi^+\setminus\Phi^+_r\}.$$
\item We compute $U(\mathfrak{n}^+)\circ f_\alpha^2$ for $\alpha\in\Phi^+_r$. 

Let $\alpha=\alpha_{p,q}$ for $p\leq r\leq q$. We consider $e_{s,t}\circ f_{p,q}^2$: it is non-zero if and only if either $s=p$, $t<q$ or $p<s$, $t=q$. By symmetry we consider the case $s=p$ and $t<q$, then $e_{p,t}\circ f_{p,q}^2=f_{t+1,q}f_{p,q}$. If $r\leq t$ then it is in the ideal generated by the elements in (G1). We assume that $t<r$, which means that $f_{t+1,q}f_{p,q}$ is of the form in (G2).

We consider $e_{i,j}\circ f_{t+1,q}f_{p,q}$: similar to the above argument, the only case we need to consider is $j=q$ and $i\geq r+1$. In this case,
$$e_{i,q}\circ f_{t+1,q}f_{p,q}=f_{t+1,i-1}f_{p,q}+f_{t+1,q}f_{p,i-1},$$
the right hand side has the form in (G2).

It remains to show that for $\alpha,\beta\in\Phi_r^+$, $\gamma\in\Phi^+$,
\begin{equation}\label{Action}
e_\gamma\circ (f_\alpha f_\beta+f_{\alpha\meet\beta}f_{\alpha\join\beta})
\end{equation}
is in the ideal generated by (G1) and (G2). According to the above arguments, it suffices to consider the case where $\alpha=\alpha_{i,j}$, $\beta=\alpha_{k,\ell}\in\Phi_r^+$ with $1\leq i<k\leq j<\ell\leq n$, then $\alpha\join\beta=\alpha_{i,\ell}$ and $\alpha\meet\beta=\alpha_{k,j}$. 

Consider the case where $\gamma=\alpha_{i,t}$. 
\begin{itemize}
\item If $t\geq r$ then the element in (\ref{Action}) is in the ideal generated by (G1). 
\item If $t<r$ then the element in (\ref{Action}) reads
$f_{t+1,j}f_{k,\ell}+f_{t+1,\ell}f_{k,j}$, which is in (G2).
\end{itemize}
In the similar way one studies the cases $\alpha=\alpha_{k,t}$, $\alpha_{t,j}$ and $\alpha_{t,\ell}$.
\end{enumerate}
\end{proof}

\begin{corollary}\label{Cor:Ideal}
The defining ideal $I_\tau^a(\varpi_r)$ is generated by the following elements:
\begin{itemize}
\item[(G1')] $f_\alpha$ for $\alpha\in N(\tau)\setminus\Phi_r^+$;
\item[(G2')] $f_\alpha f_\beta+f_{\alpha\join\beta} f_{\alpha\meet\beta}$ for $\alpha,\beta\in N(\tau)\cap\Phi_r^+$ such that both $\alpha\join\beta$ and $\alpha\meet\beta$ exist.
\end{itemize}
\end{corollary}

\begin{proof}
Since $N(\tau)$ is a rectangular subset, it suffices to observe the obvious relation 
$$I_\tau^a(\varpi_r)=I^a(\varpi_r)\cap S(\mathfrak{n}_\tau^{-,a}),$$
then apply Lemma \ref{Lem:DefIdeal} and the PBW theorem.
\end{proof}

Let $r'$ be the index of the fundamental module for $\wt{G}$ as in Proposition \ref{proposition_finalResultFundamental} by taking $\lambda=\varpi_r$. The Demazure module $V_\ttau(\wt{\varpi}_{r'})$ is a cyclic $S(\wt{\mathfrak{n}}_\ttau^+)$-module with a cyclic vector $v_{\ttau(\wt{\varpi}_{r'})}$. Let $\wt{I}_\ttau(\wt{\varpi}_{r'})$ denote its defining ideal.

We consider the composition
$$\varphi_r:S(\mathfrak{n}_\tau^{-,a})\stackrel{\sim}{\longrightarrow} S(\wt{\mathfrak{n}}_\ttau^+)\longrightarrow S(\wt{\mathfrak{n}}_\ttau^+)/\wt{I}_\ttau(\wt{\varpi}_{r'})\simeq V_{\ttau}(\wt{\varpi}_{r'}).$$

\begin{lemma}\label{Lem:Fundamental}
The map $\varphi_r$ passes through $I_\tau^a(\varpi_r)$.
\end{lemma}

\begin{proof}
We show that the generators in Corollary \ref{Cor:Ideal} are annihilated by $\varphi_r$.

Let $\alpha\in N(\tau)\setminus\Phi_r^+$. Then we have
$$\varphi_r(f_\alpha)=\Psi_r(f_\alpha)\cdot v_{\wt{\varpi}_{r'}}=\wt{e}_{\mathcal{D}(\alpha)}\cdot v_{\ttau(\wt{\varpi}_{r'})}.$$
By Corollary \ref{corollary_finalResultCombinatorics}, 
$$\pair{\ttau(\wt{\varpi}_{r'})}{\mathcal{D}(\alpha)}=\pair{\tau(\varpi_r)}{\alpha}=-\pair{\varpi_r}{\alpha}=0,$$
implying that $f_\alpha\in\ker\varphi_r$.

Let $\alpha,\beta\in N(\tau)\cap\Phi^+_r$. By Corollary \ref{corollary_latticeIsomorphism}, we have
\begin{eqnarray*}
\varphi_r(f_\alpha f_\beta+f_{\alpha\join\beta} f_{\alpha\meet\beta})&=&(\wt{e}_{\mathcal{D}(\alpha)}\wt{e}_{\mathcal{D}(\beta)}+\wt{e}_{\mathcal{D}(\alpha\join\beta)}\wt{e}_{\mathcal{D}(\alpha\meet\beta)})\cdot v_{\ttau(\wt{\varpi}_{r'})}\\
&=& \ttau((\wt{f}_{\mathcal{D}_A(\alpha)}\wt{f}_{\mathcal{D}_A(\beta)}+\wt{f}_{\mathcal{D}_A(\alpha)\join\mathcal{D}_A(\beta)}\wt{f}_{\mathcal{D}_A(\alpha)\meet\mathcal{D}_A(\beta)})\cdot v_{\wt{\varpi}_{r'}})
\end{eqnarray*}
where $\mathcal{D}_A:=i_A\circ \mathcal{D}$. Since the fundamental representation $\wt{V}(\varpi_{r'})$ is minuscule, all weight spaces have dimension one, and the above formula gives $0$.
\end{proof}

\begin{proposition}\label{Prop:Fundamental}
The map $\varphi_r$ is an isomorphism of $\n_\tau^{-,a}$-modules
$$
\varphi_r: V_\tau^{-,a}(\varpi_r) \stackrel{\sim}{\longrightarrow}  V_{\ttau}(\wt{\varpi}_{r'}).
$$
\end{proposition}
\begin{proof}
Combine Lemma \ref{Lem:Fundamental} and Proposition \ref{Prop:DimEst}.
\end{proof}

\subsection{General case}\label{Sec:GeneralCase}

We are now ready for the general case. The idea of the proof is similar to \cite[Theorem 4.1]{CLL}, but simpler for the following two reasons:
\begin{enumerate}
\item One of the goal of \cite{CLL} is to recover the defining ideal of the PBW-degenerate modules, in our case we assume this knowledge.
\item The dimension estimation (\ref{Prop:DimEst}) is established.
\end{enumerate}

\begin{proposition}\label{Prop:IsoGeneral}
Let $\lambda\in\Lambda^+$ and $\tau\in W$ be an irreducible rectangular element. There exists an isomorphism
\[
\varphi_\tau: V_\tau^{-,a}(\lambda) \stackrel{\sim}{\longrightarrow}  V_{\widetilde{\tau}}(\widetilde{\lambda}).
\]
of $\mathfrak{n}_\tau^{-,a}$-modules and hence of $N_\tau^{-,a}$-modules.
\end{proposition}

\begin{proof}
Let $\lambda=a_1\varpi_1+\cdots+a_n\varpi_n$ for $a_1,\cdots,a_n\geq 0$. Then in Proposition \ref{proposition_nablaIdealWeight} we can choose $\wt{\lambda}=a_1\wt{\varpi}_{1'}+\cdots+a_n\wt{\varpi}_{n'}$, where for $k=1,\cdots,n$, the index $k'$ is obtained from Proposition \ref{proposition_finalResultFundamental} by taking $r=k$. 

We consider the following diagram:
\[
\xymatrix{
V_{\ttau}(\tlambda) \ar@{^(->}[r]  & 
V_{\ttau}(\wt{\varpi}_{1'})^{\ts a_1}\ts\cdots\ts V_{\ttau}(\wt{\varpi}_{n'})^{\ts a_n}
\\
V_{\tau}^{-,a}(\lambda) \ar[u]^-{\varphi_\tau} \ar[r]^-{\pi} &
V_{\tau}^{-,a}(\varpi_1)^{\ts a_1}\ts\cdots\ts V_{\tau}^{-,a}(\varpi_n)^{\ts a_n} \ar[u]^-{\sim},
}
\]
where 
\begin{itemize}
\item the upper horizontal arrow is the isomorphism onto the Cartan component of the tensor product; the existence of such an isomorphism follows by the definition of Demazure modules and the irreducibility over $\C$ of the highest weight $G$--modules;
\item the right vertical arrow is the tensor product of various $\varphi_\tau$'s in Proposition \ref{Prop:Fundamental};
\item the map $\pi$ is the projection onto the Cartan component, its existence follows by the compatibility between the PBW-filtration and the tensor product; it is a surjection by the Minkowski property \cite[Lemma 3]{Fou};
\item the left vertical map $\varphi_\tau$ is defined as the composition of the other three maps, making the diagram commutative. By construction it is a surjective $\mathfrak{n}_{\tau}^{-,a}$-module morphism.
\end{itemize}
Finally note that, by Proposition \ref{Prop:DimEst}, $V_\tau^{-,a}(\lambda)$ and $V_{\widetilde{\tau}}(\widetilde{\lambda})$ have the same dimension, hence the surjective map $\varphi_\tau$ is an isomorphism.
\end{proof}

\subsection{Proof of Theorem \ref{Thm:Schubert}}

We first consider the case where $\tau\in W$ is an irreducible rectangular element.

By Proposition \ref{Prop:IsoGeneral}, the isomorphism $\varphi_\tau$ induces 
$$X_\tau^a(\lambda)\subseteq \mathbb{P}(V_\tau^{-,a}(\lambda))\simeq\mathbb{P}(V_\ttau(\tlambda))\supseteq X_{\ttau}(\tlambda).$$

We look at the open dense parts $X_\ttau(\tlambda)^\circ:=\wt{B}\cdot v_{\ttau(\tlambda)}$ and $X_\tau^a(\lambda)^\circ:=N_\tau^{-,a}\cdot v_\lambda^a$ and show that they are isomorphic.

Again by Proposition \ref{Prop:IsoGeneral}, we have the following isomorphisms:

$$\wt{B}\cdot v_{\ttau(\tlambda)}\simeq \prod_{\gamma\in \wt{N}(\ttau^{-1})}\wt{U}_\gamma\cdot v_{\ttau(\tlambda)}\simeq\prod_{\beta\in N(\tau)}U_{-\beta}\cdot v_\lambda^a\simeq N_\tau^{-,a}\cdot v_\lambda^a.$$

The $\mathbb{G}^M$-equivariancy comes from the fact that both Lie algebras $\mathfrak{n}_\tau^{-,a}$ and $\wt{\mathfrak{n}}_\ttau^+$ are abelian.

Let $\tau\in W$ be a rectangular element. According to Proposition \ref{Prop:DecompRectangular}, there exists pairwise disjoint connected subsets $I_1,\cdots,I_t\subseteq [n]$ and $A_k=N(\tau)\cap\Phi_{I_k}^+$ such that $A_k$ is a rectangular subset in $\Phi_{I_k}^+$ and $N(\tau)=A_1\sqcup\cdots\sqcup A_t$. Then there exist $\tau_1,\cdots,\tau_t$ such that $\tau_k\in W_{I_k}$ is an irreducible rectangular element and $\tau=\tau_1\cdots\tau_t$ where $W_{I_k}$ is the Weyl group of $\Phi_{I_k}^+$. To terminate the proof of Theorem \ref{Thm:Schubert}, it suffices to apply the above irreducible rectangular case to each $\tau_1,\cdots,\tau_t$ and the sub-root systems separately.

\section{Consequences and examples}\label{sec-cons}

\begin{corollary}
Let $\tau\in W$ be a rectangular element and $\lambda\in\Lambda^+$. The degenerate Schubert variety $X_\tau^a(\lambda)$ is projectively normal, has rational singularities and is Frobenius split.
\end{corollary}

\begin{example}\label{Ex:w0}
Let $\tau=w_0\in W$ be the longest element in the Weyl group, which is an irreducible rectangular element. It is easy to verify that the Weyl group element $\ttau\in \wt{W}$ coincides with the permutation $\sigma$ given in \cite{CL}. In this case, our Theorem \ref{Thm:Schubert} coincides with \cite[Theorem 1.2]{CL}.
\end{example}

\begin{remark}\label{Rmk:NonRect}
We consider $\g=\mathfrak{sl}_4$ and $s_1s_3s_2\in\mathfrak{S}_4$. Our construction provides the element $s_5s_4s_3s_5s_4\in\mathfrak{S}_6$.
It is easy to check, that for the first (and the third) fundamental weight of $\mathfrak{sl}_4$, there is no dominant integral weight $\tilde{\lambda}$ for $\mathfrak{sl}_6$ such that $V^{-,a}_{s_1s_3s_2}(\varpi_1)$ is isomorphic to the Demazure module $V_{s_5s_4s_3s_5s_4}(\widetilde{\lambda})$.
\end{remark}

\begin{remark}
Our approach is different to the one in \cite{BL}. There the authors study which ``classical'' Schubert variety in a flag variety of type $\typeA$ stays irreducible and/or is still a Schubert variety when degenerated following a certain degeneration of the flag variety to the degenerate flag variety for type $\typeA$. But this kind of degeneration of Schubert varieties is different from the degenerate Schubert varieties considered in this paper, for example our degenerate Schubert varieties are all irreducible by definition.
\end{remark}

\end{document}